\documentclass{amsart}

\usepackage{eurosym}
\usepackage{amssymb}
\usepackage{lscape}
\usepackage{amsmath}
\usepackage{amsmath}
\usepackage{amsfonts}
\usepackage{amsthm}
\usepackage{amssymb}
\usepackage{amscd}
\usepackage{mathrsfs}
\usepackage{pb-diagram,pb-xy}
\usepackage[cmtip,arrow]{xy}
\usepackage{graphicx}

\xyoption{all}
\newtheorem{theorem}{\quad Theorem}[section]
\theoremstyle{plain}

\newtheorem{corollary}[theorem]{Corollary}

\newtheorem{definition}[theorem]{Definition}

\newtheorem{lemma}[theorem]{Lemma}

\newtheorem{proposition}[theorem]{Proposition}

\numberwithin{equation}{section}

\begin{document}
\title[The Maps Category]{Tilting theory and functor categories III.\\
The Maps Category.	}
\author{R. Mart\'{\i}nez-Villa}
\address{Instituto de Matem\'{a}ticas UNAM, Unidad Morelia, Mexico}
\email{mvilla@matmor.unam.mx}
\thanks{}
\author{M. Ortiz-Morales}
\address{Instituto de Matem\'{a}ticas UNAM, Unidad Morelia, Mexico}
\email{mortiz@matmor.unam.mx}
\urladdr{http://www.matmor.unam.mx}
\thanks{The second author thanks CONACYT for giving him financial support during his graduate studies }
\date{January 21, 2010}
\dedicatory{}
\thanks{This paper is in final form and no version of it will be submitted
for publication elsewhere.}

\begin{abstract}
In this paper we continue the project of generalizing tilting theory to the
category of contravariant functors $\mathrm{Mod}(\mathcal{C})$, from a
skeletally small preadditive category $\mathcal{C}$ to the category of
abelian groups, initiated in [17]. In [18] we introduced the notion of a
a generalized tilting category $\mathcal{T}$, and extended Happel's theorem
to $\mathrm{Mod}(\mathcal{C})$. We proved that there is an equivalence of
triangulated categories $D^{b}(\mathrm{Mod}(C))\cong D^{b}(\mathrm{Mod}(%
\mathcal{T}))$. In the case of dualizing varieties, we proved a version of
Happel's  theorem for the categories of finitely presented functors. We also proved
in this paper, that there exists a relation between covariantly finite coresolving
categories, and generalized tilting categories. Extending theorems for artin
algebras proved in [4], [5]. In this article we consider the category of
maps, and relate tilting categories in the category of functors, with relative
tilting in the category of maps. Of special interest is the category $%
\mathrm{mod}(\mathrm{mod}\Lambda)$ with $\Lambda$ an artin algebra.

\end{abstract}

\maketitle

\section{\protect\bigskip Introduction and basic results}

This is the last article in a series of three in which, having in
mind applications to the category of functors from subcategories of modules
over a finite dimensional algebra to the category of abelian groups, we
generalize tilting theory, from rings to functor categories.

In the first paper [17] we generalized classical tilting to the category of
contravariant functors from a preadditive skeletally small category $%
\mathcal{C}$, to the category of abelian groups and generalized Bongartz's
proof [10] of Brenner-Butler's theorem [11]. We then applied the theory so
far developed, to the study of locally finite infinite quivers with no
relations, and computed the Auslander-Reiten components of infinite Dynkin
diagrams. Finally, we applied our results to calculate the Auslander-Reiten
components of the category of Koszul functors (see [19], [20], [21])
on a regular component of a finite dimensional algebra over a field. These
results generalize the theorems on the preprojective algebra obtained in
[15].

Following [12], in [18] we generalized the proof of Happel's theorem
given by Cline, Parshall and Scott: given a generalized tilting subcategory $%
\mathcal{T}$ of $\mathrm{Mod}(\mathcal{C})$, the derived categories of
bounded complexes $D^{b}(\mathrm{Mod}(\mathcal{C}))$ and $D^{b}(\mathrm{Mod}(%
\mathcal{T}))$ are equivalent, and we discussed a partial converse [14]. We
also saw that for a dualizing variety $\mathcal{C}$ and a tilting
subcategory $\mathcal{T}\subset \mathrm{mod}(\mathcal{C})$ with
pseudokerneles, the categories of finitely presented functors $\mathrm{mod}(%
\mathcal{C})$ and $\mathrm{mod}(\mathcal{T})$ have equivalent derived
bounded categories, $D^{b}(\mathrm{mod}(\mathcal{C}))\cong D^{b}(\mathrm{mod}%
(\mathcal{T}))$. Following closely the results for artin algebras obtained
in [3], [4], [5], by Auslander, Buchweits and Reiten, we end the paper
proving that for a Krull-Schmidt dualizing variety $\mathcal{C}$, there are
analogous relations between covariantly finite subcategories and generalized
tilting subcategories of $\mathrm{mod}(\mathcal{C})$.

This paper is dedicated to study tilting subcategories of $\mathrm{mod}(%
\mathcal{C})$. In order to have a better understanding of these categories,
we use the relation between the categories $\mathrm{mod}(\mathcal{C})$ and the
category of maps, $\mathrm{maps}(\mathcal{C})$, given by Auslander in [1].
Of special interest is the case when $\mathcal{C}$ is the category of
finitely generated left $\Lambda $-modules over an artin algebra $\Lambda $,
since in this case the category $\mathrm{maps}(\mathcal{C})$ is equivalent
to the category of finitely generated $\Gamma $modules, $\mathrm{mod}(\Gamma
), $ over the artin algebra of triangular matrices $\Gamma =\left(
\begin{array}{cc}
\Lambda & 0 \\
\Lambda & \Lambda%
\end{array}%
\right) $. In this situation, tilting subcategories on $\mathrm{mod}(\mathrm{%
mod}(\Lambda) \mathcal{)}$ will correspond to relative tilting subcategories
of $\mathrm{mod}(\Gamma )$, which in principle, are easier to compute.

The paper consists of three sections:

In the first section we establish the notation and recall some basic
concepts. In the second one, for a variety of annuli with pseudokerneles $%
\mathcal{C}$, we prove that generalized tilting subcategories of $\mathrm{mod%
}(\mathcal{C})$ are in correspondence with relative tilting subcategories of
$\mathrm{maps}(\mathcal{C})$ [9]. In the third section, we explore the
connections between $\mathrm{mod}\ \Gamma $, with $\Gamma =\left(
\begin{array}{cc}
\Lambda & 0 \\
\Lambda & \Lambda%
\end{array}%
\right) $ and the category $\mathrm{mod}(\mathrm{mod}(\Lambda) )$. We compare
the Auslander-Reiten sequences in $\mathrm{\mathrm{mod}}(\Gamma )$ with
Auslander-Reiten sequences in $\mathrm{mod}(\mathrm{mod}(\Lambda) )$. We end
the paper proving that, some important subcategories of $\mathrm{mod}(%
\mathcal{C})$ related with tilting, like: contravariantly, covariantly,
functorially finite [see 18], correspond to subcategories of $\mathrm{maps}(%
\mathcal{C})$ with similar properties.

\subsection{Functor Categories}

In this subsection we will denote by $\mathcal{C}$ an arbitrary skeletally
small pre additive category, and $\mathrm{Mod}(\mathcal{C})$ will be the
category of contravariant functors from $\mathcal{C}$ to the category of
abelian groups. The subcategory of $\mathrm{Mod}(\mathcal{C})$ consisting of
all finitely generated projective objects, $\mathfrak{p}(\mathcal{C})$, is a
skeletally small additive category in which idempotents split, the functor $%
P:\mathcal{C}\rightarrow \mathfrak{p}(\mathcal{C})$, $P(C)=\mathcal{C}(-,C)$%
, is fully faithful and induces by restriction $\mathrm{res}:\mathrm{Mod}(%
\mathfrak{p}(\mathcal{C}))\rightarrow \mathrm{Mod}(\mathcal{C})$, an
equivalence of categories. For this reason, we may assume that our
categories are skeletally small, additive categories, such that idempotents
split. Such categories were called \textbf{annuli varieties} in [2], for
short, varieties.

To fix the notation, we recall known results on functors and categories that
we use through the paper, referring for the proofs to the papers by
Auslander and Reiten [1], [4], [5].

Given a category $\mathcal{C}$ we will write for short, $\mathcal{C}(-,?)$
instead of $\mathrm{Hom}_{\mathcal{C}}(-,?)$ and when it is clear from the context we
use just $(-,?).$

\begin{definition}
Given a variety $\mathcal{C}$, we say $\mathcal{C}$ has \textbf{pseudokernels%
}; if given a map $f : C_1\rightarrow C_0$, there exists a map $g : C_2
\rightarrow C_1$ such that the sequence of representable functors $\mathcal{C%
}(-, C_2 )\xrightarrow{(-,g)}\mathcal{C}( -,C_1 )\xrightarrow{(-,f)}
\mathcal{C}(-, C_0 )$ is exact.
\end{definition}

A functor $M$ is \textbf{finitely presented}; if there exists an exact
sequence
\begin{equation*}
\mathcal{C}( -,C_1 )\rightarrow \mathcal{C}(-, C_0 )\rightarrow M\rightarrow 0
\end{equation*}

We denote by $\mathrm{mod}(\mathcal{C})$ the full subcategory of $\mathrm{Mod%
}(\mathcal{C})$ consisting of finitely presented functors. It was proved in
[1] $\mathrm{mod}(C)$ is abelian, if and only if, $\mathcal{C}$ has
pseudokernels.

\subsection{ Krull-Schmidt Categories}

We start giving some definitions from [6].

\begin{definition}
Let $R$ be a commutative artin ring. An $R$-variety $\mathcal{C}$, is a
variety such that $\mathcal{C}(C_{1},C_{2})$ is an $R$-module, and
composition is $R$-bilinear. Under these conditions $\mathrm{\mathrm{Mod}}(%
\mathcal{C})$ is an $R$-variety, which we identify with the category of
contravariant functors $(\mathcal{C}^{op},\mathrm{Mod}(R))$.

An $R$-variety $\mathcal{C}$ is $\mathrm{Hom}$-\textbf{finite}, if for each
pair of objects $C_{1},C_{2}$ in $\mathcal{C},$ the $R$-module $\mathcal{C}%
(C_{1},C_{2})$ is finitely generated. We denote by $(\mathcal{C}^{op},%
\mathrm{mod}(R))$, the full subcategory of $(\mathcal{C}^{op},\mathrm{%
\mathrm{Mod}}(R))$ consisting of the $\mathcal{C}$-modules such that; for
every $C$ in $\mathcal{C}$ the $R$-module $M(C)$ is finitely generated. The
category $(\mathcal{C}^{op},\mathrm{mod}(R))$ is abelian and the inclusion $(%
\mathcal{C}^{op},\mathrm{mod}(R))\rightarrow (\mathcal{C}^{op},\mathrm{%
\mathrm{Mod}}(R))$ is exact.
\end{definition}

The category $\mathrm{mod}(C)$ is a full subcategory of $(\mathcal{C}^{op},%
\mathrm{mod}(R))$. The functors $D:(\mathcal{C}^{op},\mathrm{mod}%
(R))\rightarrow (\mathcal{C},\mathrm{mod}(R))$, and $D:(\mathcal{C},\mathrm{%
mod}(R))\rightarrow (\mathcal{C}^{op},\mathrm{mod}(R))$, are defined as
follows: for any $C$ in $\mathcal{C}$, $D(M)(C)=\mathrm{Hom}%
_{R}(M(C),I(R/r)) $, with $r$ the Jacobson radical of $R$, and $I(R/r)$ is
the injective envelope of $R/r$. The functor $D$ defines a duality between $(%
\mathcal{C},\mathrm{mod}(R))$ and $(\mathcal{C}^{op},\mathrm{mod}(R))$. If $%
\mathcal{C}$ is an $\mathrm{Hom}$-finite $R$-category and $M$ is in $\mathrm{%
mod}(\mathcal{C})$, then $M(C)$ is a finitely generated $R$-module and it is
therefore in $\mathrm{mod}(R)$.

\begin{definition}
An $\mathrm{Hom}$-finite $R$-variety $\mathcal{C}$ is \textbf{dualizing}, if
the functor
\begin{equation*}
D:(\mathcal{C}^{op},\mathrm{mod}(R))\rightarrow (\mathcal{C},\mathrm{mod}(R))
\end{equation*}
induces a duality between the categories $\mathrm{mod}(\mathcal{C})$ and $%
\mathrm{mod}(\mathcal{C}^{op}).$
\end{definition}

It is clear from the definition that for dualizing categories $\mathcal{C}$
the category $\mathrm{mod}(\mathcal{C})$ has enough injectives.

To finish, we recall the following definition:

\begin{definition}
An additive category $\mathcal{C}$ is \textbf{Krull-Schmidt}, if every
object in $\mathcal{C}$ decomposes in a finite sum of objects whose
endomorphism ring is local.
\end{definition}

In [18 Theo. 2] we see that for a dualizing Krull-Schmidt variety the
finitely presented functors have projective covers.

\begin{theorem}
Let $\mathcal{C}$ a dualizing Krull-Schmidt $R$-variety. Then $\mathrm{mod}(%
\mathcal{C})$ is a dualizing Krull-Schmidt variety.
\end{theorem}

\subsection{Contravariantly finite categories}

[4] Let $\mathscr X$ be a subcategory of $\mathrm{mod}(\mathcal{C})$,
which is closed under summands and isomorphisms. A morphism $f:X\rightarrow M
$ in $\mathrm{mod}(\mathcal{C})$, with $X$ in $\mathscr X$, is a \emph{right
}$\mathscr X$-\emph{approximation} of $M$, if $(-,X)_{\mathscr X}%
\xrightarrow{(-,h)_\mathscr X}(-,M)_{\mathscr X}\rightarrow 0$ is an exact
sequence, where $(-,?)_{\mathscr X}$  denotes the restriction of $(-,?)$ to
the category  $\mathscr X$. Dually, a morphism $g:M\rightarrow X$, with $X$
in $\mathscr X$, is a \emph{left} $\mathscr X$-\emph{approximation} of $M$,
if $(X,-)_{\mathscr X}\xrightarrow{(g,-)_\mathscr X}(M,-)_{\mathscr %
X}\rightarrow 0$ is exact.

A subcategory $\mathscr X$ of $\mathrm{mod}(\mathcal{C})$ is called \emph{%
contravariantly} (\emph{covariantly}) finite in $\mathrm{mod}(\mathcal{C})$,
if every object $M$ in $\mathrm{mod}(\mathcal{C})$ has a right (left) $%
\mathscr X$-approximation; and \emph{functorially finite}, if it is both
contravariantly and covariantly finite.

A subcategory $\mathscr X$ of $\mathrm{mod}(\mathcal{C})$ is \emph{resolving}
(\emph{coresolving}), if it satisfies the following three conditions: (a) it
is closed under extensions, (b) it is closed under kernels of epimorphisms
(cokernels of monomorphisms), and (c) it contains the projective (injective)
objects.

\subsection{Relative Homological Algebra and Frobenius Categories}

In this subsection we recall some results on relative homological algebra
introduced by Auslander and Solberg in [9],[see also 14, 23].

Let $\mathcal{C}$ be an additive category which is embedded as a full
subcategory of an abelian category $\mathcal{A}$, and suppose that $\mathcal{%
C}$ is closed under extensions in $\mathcal{A}$. Let $\mathcal{S}$ be a
collection of exact sequences in $\mathcal{A}$
\begin{equation*}
0\rightarrow X\xrightarrow{f}Y\xrightarrow{g}Z\rightarrow 0
\end{equation*}%
$f$ is called an \emph{admissible monomorphism}, and $g$ is called an \emph{%
admissible epimorphism}. A pair $(\mathcal{C},\mathcal{S})$ is called an
\emph{exact category} provided that: (a) Any split exact sequence whose
terms are in $\mathcal{C}$ is in $\mathcal{S}$. (b) The composition of
admissible monomorphisms (resp., epimorphisms) is an admissible monomorphism
(resp., epimorphism). (c) It is closed under pullbacks (pushouts) of
admissible epimorphisms (admissible monomorphisms).

Let $(\mathcal{C},\mathcal{S})$ be an exact subcategory of an abelian
category $\mathcal{A}$. Since the collection $\mathcal{S}$ is closed under
pushouts, pullbacks and Baer sums, it gives rise to a subfunctor $F$ of the
additive bifunctor $\mathrm{Ext}_{\mathcal{C}}^{1}(-,-):\mathcal{C}%
\times \mathcal{C}^{op}\rightarrow \mathbf{Ab}$ [9]. Given such a
functor $F$, we say that an exact sequence $\eta :0\rightarrow A\rightarrow
B\rightarrow C\rightarrow 0$ in $\mathcal{C}$ is $F$-exact, if $\eta $ is in
$F(C,A)$, we will write some times $\mathrm{Ext}_{F}^{1}(-,?)$ instead of $F(-,?)$.
An object $P$ in $\mathcal{C}$ is $F$-projective, if for each $F$-exact
sequence $0\rightarrow A\rightarrow B\rightarrow C\rightarrow 0$, the
sequence $0\rightarrow (P,N)\rightarrow (P,E)\rightarrow (P,M)\rightarrow 0$
is exact. Analogously we have the definition of an $F$-injective object.

If for any object $C$ in $\mathcal{C}$ there is an $F$-exact sequence $%
0\rightarrow A\rightarrow P\rightarrow C\rightarrow 0$, with $P$ an $F$%
-projective, then we say $(\mathcal{C},\mathcal{S})$ has enough $F$-
projectives. Dually, if for any object $C$ in $\mathcal{C}$ there is an $F$%
-exact sequence $0\rightarrow C\rightarrow I\rightarrow A\rightarrow 0$,
with $I$ an $F$-injective, then $(\mathcal{C},\mathcal{S})$ has enough $F-$
injectives.

An exact category $(\mathcal{C},\mathcal{S})$ is called \emph{Frobenius}, if
the category $(\mathcal{C},\mathcal{S})$ has enough $F$-projectives and enough $F$%
-injectives and they coincide.

Let $F$ be a subfunctor of $\mathrm{Ext}^1_{\mathcal{C}}(-,-)$. Suppose $F$
has enough projectives. Then for any $C$ in $\mathcal{C}$ there is an exact
sequence in $\mathcal{C}$ of the form
\begin{equation*}
\cdots P_n\xrightarrow{d_n}P_{n-1}\xrightarrow{d_{n-1}}\cdots\rightarrow P_1%
\xrightarrow{d_1} P_0\xrightarrow{d_0} C\rightarrow 0
\end{equation*}
where $P_i$ is $F$-projective for $i\ge 0$ and $0\rightarrow \mathrm{Im}%
d_{i+1}\rightarrow P_i\rightarrow \mathrm{Im}d_{i}\rightarrow 0$ is $F$%
-exact for all $i\ge 0$. Such sequence is called an $F$-\emph{exact
projective resolution}. Analogously we have the definition of an $F$-\emph{%
exact injective resolution}.

When $(\mathcal{C},\mathcal{S})$ has enough $F$-injectives $($enough $F$-
projectives), using $F$-exact injective resolutions (respectively, $F$-exact
projective resolutions), we can prove that for any object $C$ in $\mathcal{C}
$, ($A$ in $\mathcal{C}$ ), there exists a right derived functor of $\mathrm{%
Hom}_{\mathcal{C}}(C,-)$ ( $\mathrm{Hom}_{\mathcal{C}}(-,A)$ ).

We denote by $\mathrm{Ext}_{F}^{i}(C,-)$ the right derived functors of $%
\mathrm{Hom}_{\mathcal{C}}(C,-)$ and by $\mathrm{Ext}_{F}^{i}(-,A)$ the
right derived functors of $\mathrm{Hom}_{\mathcal{C}}(-,A)$.

\section{The maps category, $\mathrm{maps}(\mathcal{C})$}

In this section $\mathcal{C}$ is an annuli variety with pseudokerneles. We
will study tilting subcategories of $\mathrm{mod}(\mathcal{C})$ via the
equivalence of categories between the maps category, module the homotopy
relation, and the category of functors, $\mathrm{mod}(\mathcal{C})$, given
by Auslander in [1]. We will provide $\mathrm{maps}(\mathcal{C})$ with a
structure of exact category such that, tilting subcategories of $\mathrm{mod}%
(\mathcal{C})$ will correspond to relative tilting subcategories of $\mathrm{%
maps}(\mathcal{C})$. We begin the section recalling concepts and results
from [1], [14] and [23].

The objects in $\mathrm{maps}(\mathcal{C})$ are morphisms $%
(f_{1},A_{1},A_{0}):A_{1}\xrightarrow{f_1}A_{0}$, and the maps are pairs $%
(h_{1},h_{0}):(f_{1},A_{1},A_{0})\rightarrow (g_{1},B_{1},B_{0})$, such that
the following square commutes
\begin{equation*}
\begin{diagram}\dgARROWLENGTH=1em
\node{A_1}\arrow{e,t}{f_1}\arrow{s,l}{h_1}
\node{A_0}\arrow{s,l}{h_0}\\ \node{B_1}\arrow{e,t}{g_1} \node{B_0}
\end{diagram}
\end{equation*}%
We say that two maps $(h_{1},h_{0})$, $(h_{1}^{\prime },h_{0}^{\prime
}):(f_{1},A_{1},A_{0})\rightarrow (g_{1},B_{1},B_{0})$ are homotopic, if
there exist a morphisms $s:A_{0}\rightarrow B_{1}$ such that $%
h_{0}-h_{0}^{\prime }=g_{1}s$. Denote by $\underline{\mathrm{maps}}(\mathcal{%
C})$ the category of maps modulo the homotopy relation. It was proved in
[1] that the categories $\underline{\mathrm{maps}}(\mathcal{C})$ and $%
\mathrm{mod}(\mathcal{C})$ are equivalent. The equivalence is given by a
functor $\underline{\varPhi}:\underline{\mathrm{maps}}(\mathcal{C}%
)\rightarrow \mathrm{mod}(\mathcal{C})$ induced by the functor $\varPhi:%
\mathrm{map}(\mathcal{C})\rightarrow \mathrm{mod}(\mathcal{C})$ given by
\begin{equation*}
\varPhi(A_{1}\xrightarrow{f_1}A_{0})=\mathrm{Coker}((-,A_{1})%
\xrightarrow{(-,f_1)}(-,A_{0}))\text{.}
\end{equation*}

The category $\mathrm{maps}(\mathcal{C})$ is not in general an exact
category, we will use instead the exact category $P^{0}(\mathcal{A})$ of
projective resolutions, which module the homotopy relation, is equivalent to
$\underline{\mathrm{maps}}(\mathcal{C})$.

Since we are assuming $\mathcal{C}$ has pseudokerneles, the category $%
\mathcal{A}=\mathrm{mod}(\mathcal{C})$ is abelian. We can consider the
categories of complexes $C(\mathcal{A})$, and its subcategory $C^{-}(%
\mathcal{A})$, of bounded above complexes, both are abelian. Moreover, if we
consider the class of exact sequences $\mathcal{S}$: $0\rightarrow {L{.}}%
\xrightarrow{j}M{.}\xrightarrow{\pi}N{.}\rightarrow 0$, such that, for every
$k$, the exact sequences $0\rightarrow L_{k}\xrightarrow{j_k}M_{k}%
\xrightarrow{\pi_k}N_{k}\rightarrow 0$ split, then $(\mathcal{S},C(\mathcal{A%
}))$, $(\mathcal{S},C^{-}(\mathcal{A}))$ are exact categories with enough
projectives, in fact they are both Frobenius. In the first case the
projective are summands of complexes of the form:
\begin{equation*}
\cdots B_{k+2}\coprod B_{k+1}\xrightarrow{\left[%
\begin{matrix} 0&1\\ 0 &0 \end{matrix}\right]}B_{k+1}\coprod B_{k}%
\xrightarrow{\left[\begin{matrix} 0&1\\ 0 &0 \end{matrix}\right]}%
B_{k}\coprod B_{k-1}\cdots
\end{equation*}

In the second case of the form:

\begin{equation*}
\cdots B_{k+3}\coprod B_{k+2}\xrightarrow{\left[\begin{matrix}
0&1\\ 0&0 \end{matrix}\right]}B_{k+2}\coprod B_{k+1}\xrightarrow{\left[%
\begin{matrix} 0&1\\ 0 &0 \end{matrix}\right]}B_{k+1}\coprod B_{k}%
\xrightarrow{\left[\begin{matrix} 0&1 \end{matrix}\right]}B_{k}\rightarrow 0
\end{equation*}

If we denote by $\underline{C^{-}}(\mathcal{A})$ the stable category, it is
well known [23], [14], that the homotopy category ${K}^{-}(\mathcal{A})$ and
$\underline{C^{-}}(\mathcal{A})$ are equivalent.

Now, denote by $P^{0}(\mathcal{A})$ the full subcategory of $C^{-}(\mathcal{A%
})$ consisting of projective resolutions, this is, complexes of projectives $%
P.$:
\begin{equation*}
\cdots  P_{k}\rightarrow P_{k-1}\rightarrow \cdots \rightarrow
P_{1}\rightarrow P_{0}\rightarrow 0
\end{equation*}%
such that $H^{i}(P.)=0$ for $i\neq 0$. Then we have the following:

\begin{proposition}
The category $P^{0}(\mathcal{A})$ is closed under extensions and kernels of
epimorphisms.
\end{proposition}

\begin{proof}
If $0\rightarrow P.\rightarrow E.\rightarrow Q.\rightarrow 0$ is an exact
sequence in $P^{0}(\mathcal{A})$, then $0\rightarrow P_{j}\rightarrow
E_{j}\rightarrow Q_{j}\rightarrow 0$ is a splitting exact sequence in $%
\mathcal{A}$ with $P_{j}$, $Q_{j}$ projectives, hence $E_{j}$ is also
projective. By the long homology sequence we have the exact sequence: $%
\cdots \rightarrow H^{i}(P.)\rightarrow H^{i}(E.)\rightarrow
H^{i}(Q.)\rightarrow H^{i-1}(P.)\rightarrow \cdots $, with $%
H^{i}(P.)=H^{i}(Q.)=0$, for $i\neq 0$. This implies $E.\in P^{0}(\mathcal{A}%
) $.

Now, let $0\rightarrow T.\rightarrow Q.\rightarrow P.\rightarrow 0$ be an
exact sequence with $Q.$, $P.$ in $P^{0}(\mathcal{A})$. This implies that
for each $k$, $0\rightarrow T_{k}\rightarrow Q_{k}\rightarrow
P_{k}\rightarrow 0$ is an exact and splittable sequence, hence each $T_{k}$
is projective and, by the long homology sequence, we have the following
exact sequence
\begin{equation*}
\cdots \rightarrow H^{1}(T.)\rightarrow
H^{1}(Q.)\rightarrow H^{1}(P.)\rightarrow H^{0}(T.)\rightarrow
H^{0}(Q.)\rightarrow H^{0}(P.)\rightarrow 0
\end{equation*}%
with $H^{i+1}(P.)=H^{i}(Q.)=0$ for $i\geq 1$. This implies $H^{i}(T.)=0$,
for $i\neq 0$.
\end{proof}

If $\mathcal{S}_{P^{0}(\mathcal{A})}$ denotes the collection of exact
sequences with objects in $P^{0}(\mathcal{A})$, then $(P^{0}(\mathcal{A}),%
\mathcal{S}_{P^{0}(\mathcal{A})})$ is an exact subcategory of $(C^{-}(%
\mathcal{A}),\mathcal{S})$. The category $P^{0}(\mathcal{A})$ has enough
projectives, they are the complexes of the form:

\begin{equation}\label{relativeproj}
\cdots \rightarrow P_{3}\coprod P_{2}\xrightarrow{\left[\begin{matrix} 0&1\\
0&0 \end{matrix}\right]}P_{2}\coprod P_{1}\xrightarrow{\left[\begin{matrix}
0&1\\ 0&0 \end{matrix}\right]}P_{1}\coprod P_{0}\rightarrow 0
\end{equation}

Denote by $R^{0}(\mathcal{A})$ the category $P^{0}(\mathcal{A})$ module the
homotopy relation. This is: $R^{0}(\mathcal{A})$ is a full subcategory of $%
\underline{\mathcal{C}^{-}}(\mathcal{A})=K^{-}(\mathcal{A})$. It is easy to
check that $R^{0}(\mathcal{A})$ is the category with objects in $P^{0}(%
\mathcal{A})$ and maps the maps of complexes, module the maps that factor
through a complex of the form:

\begin{equation*}
\cdots \rightarrow P_{3}\coprod P_{2}\xrightarrow{\left[\begin{matrix} 0&1\\
0&0 \end{matrix}\right]}P_{2}\coprod P_{1}\xrightarrow{\left[\begin{matrix}
0&1\\ 0&0 \end{matrix}\right]}P_{1}\coprod P_{0}\xrightarrow{\left[%
\begin{matrix} 0&1 \end{matrix}\right]}P_{0}\rightarrow 0
\end{equation*}

We have the following:

\bigskip

\begin{proposition}
There is a functor $\varPsi:P^{0}(\mathcal{A})\rightarrow \mathrm{maps}(%
\mathcal{C})$ which induces an equivalence of categories $\underline{\varPsi}%
:R^{0}(\mathcal{A})\rightarrow \underline{\mathrm{maps}}(\mathcal{C})$ given
by:
\begin{equation*}
\varPsi(P.)=\varPsi(\cdots \rightarrow (-,A_{2})\xrightarrow{(-,f_2)}%
(-,A_{1})\xrightarrow{(-,f_1)}(-,A_{0})\rightarrow 0)=A_{1}\xrightarrow{f_1}%
A_{0}
\end{equation*}
\end{proposition}

\begin{proof}
Since $\mathcal{C}$ has pseudokerneles, any map $A_{1}\xrightarrow{f_1}A_{0}$
induces an exact sequence
\begin{equation*}
(-,A_{n})\xrightarrow{(-,f_n)}(-,A_{n-1})\rightarrow \cdots \rightarrow
(-,A_{2})\xrightarrow{(-,f_2)}(-,A_{1})\xrightarrow{(-,f_1)}%
(-,A_{0})
\end{equation*}%
and $\varPsi$ is clearly dense. Let $(-,\varphi ):P.\rightarrow Q.$ be a map
of complexes in $P^{0}(\mathcal{A})$:
\begin{equation}
\begin{diagram}\dgARROWLENGTH=1em \node{\cdots}\arrow{e}
\node{(-,A_2)}\arrow{e,t}{(-,f_2)}\arrow{s,l}{(-,\varphi_2)}
\node{(-,A_1)}\arrow{e,t}{(-,f_1)}\arrow{s,l}{(-,\varphi_1)}
\node{(-,A_0)}\arrow{e}\arrow{s,l}{(-,\varphi_0)} \node{0}\\
\node{\cdots}\arrow{e} \node{(-,B_2)}\arrow{e,t}{(-,g_2)}
\node{(-,B_1)}\arrow{e,t}{(-,g_1)} \node{(-,B_0)}\arrow{e} \node{0}
\end{diagram}  \label{maps1}
\end{equation}

If $\varPsi(P.\xrightarrow{(-,\varphi)}Q.)$ is homotopic to zero, then we
have a map $s_{0}:A_{0}\rightarrow A_{1}$ such that $g_{0}s_{0}=\varphi _{0}$%
:
\begin{equation*}
\begin{diagram}\dgARROWLENGTH=1em
\node{A_1}\arrow{e,t}{f_1}\arrow{s,l}{\varphi_1}
\node{A_0}\arrow{s,l}{\varphi_0}\arrow{sw,t}{s_0}\\
\node{B_1}\arrow{e,t}{g_1} \node{B_0} \end{diagram}
\end{equation*}%
and $s_{0}$ lifts to a homotopy $s:P.\rightarrow Q.$. Conversely, any
homotopy $s:P.\rightarrow Q.$ induces an homotopy in $\mathrm{maps}(\mathcal{%
C})$. Then $\varPsi$ is faithful.

If $\varPsi(P.)=(f_{1},A_{1},A_{0})$, $\varPsi(Q.)=(g_{1},B_{1},B_{0})$ and $%
(h_{0},h_{1}):\varPsi(P.)\rightarrow \varPsi(Q.)$ is a map in $\mathrm{maps}(%
\mathcal{C})$, then $(h_{0},h_{1})$ lifts to a map $(-,h)={(-,h_{i})}%
:P.\rightarrow Q.$, and $\varPsi$ is full.
\end{proof}

\begin{corollary}
There is an equivalence of categories $\Theta :R^{0}(\mathcal{A})\rightarrow
\mathrm{mod}(\mathcal{C})$ given by $\underline{\varTheta}=\underline{\varPhi%
}\underline{\varPsi}$, with $\varTheta=\varPhi\varPsi$.
\end{corollary}

\begin{proposition}
\label{omegaP} Let $P.$ be an object in $P^{0}(\mathcal{A})$, denote by $%
\mathrm{rpdim}P.$ the relative projective dimension of $P.$ , and by $%
\mathrm{pdim}\varTheta(P.)$ the projective dimension of $\varTheta(P.)$.
Then we have $\mathrm{rpdim}P.=\mathrm{pdim}\varTheta(P.)$. Moreover, if $%
\Omega ^{i}(P.)$ is the relative syzygy of $P.$, then for all $i\geq 0$, we
have $\Omega ^{i}(\varTheta(P.))=\varTheta(\Omega ^{i}(P.))$.
\end{proposition}

\begin{proof}
Let $P.$ be the complex resolution
\begin{equation*}
\text{0}\rightarrow \text{(-,}A_{n})\xrightarrow{(-,f_n)}(\text{-,}%
A_{n-1})\rightarrow \cdots \rightarrow (\text{-,}A_{2})\xrightarrow{(-,f_2)}(%
\text{-,}A_{1})\xrightarrow{(-,f_1)}\text{(-,}A_{0}\text{)}\rightarrow \text{%
0}
\end{equation*}%
and $M=\mathrm{Coker}(-,f_{1})$, then $\mathrm{pdim}M\leq n$.

Now, consider the following commutative diagram
\begin{equation*}
\begin{diagram}\dgARROWLENGTH=.9em
\node{} \node{} \node{}
\node{0}\arrow{e,t}{}\arrow{s,=,-} \node{A_n}\arrow{s,l}{\big(
^{1}_{0}\big)}\\ \node{} \node{} \node{} \node{A_n}\arrow{e,t}{-\big(
^{1}_{0}\big)}\arrow{s,..,-} \node{A_n\coprod A_{n-1}}\arrow{s,..,-}\\
\node{} \node{} \node{A_n}\arrow{s,r}{\big( ^{1}_{f_n}\big)} \arrow{e,..,-}
\node{A_3}\arrow{s,r}{\big( ^{1}_{f_3}\big)}\arrow{e,t}{f_3}
\node{A_2}\arrow{s,r}{\big( ^{1}_{f_2}\big)}\\ \node{}
\node{A_n}\arrow{s,=}\arrow{e,t}{-\big( ^{1}_{0}\big)} \node{A_n\coprod
A_{n-1}}\arrow{s,r}{(-f_n\; 1)}\arrow{e,..,-} \node{A_3\coprod
A_2}\arrow{s,r}{(-f_3\; 1)}\arrow{e,t}{\big( _{0\; 0}^{0\; 1} \big)}
\node{A_2\coprod A_1}\arrow{s,r}{(f_2\; -1)}\\ \node{}
\node{A_n}\arrow{s,r}{\big( ^{1}_{f_n}\big)} \arrow{e,t}{f_n}
\node{A_{n-1}}\arrow{s,r}{\big( ^{1}_{f_{n-1}}\big)}\arrow{e,..,-}
\node{A_2}\arrow{s,r}{\big( ^{1}_{f_2}\big)}\arrow{e,t}{f_2}
\node{A_1}\arrow{s,r}{\big( ^{1}_{f_1}\big)}\\
\node{A_n}\arrow{s,=}\arrow{e,t}{-\big( ^{1}_{0}\big)} \node{A_n\coprod
A_{n-1}}\arrow{s,r}{(-f_n\; 1)}\arrow{e,t}{\big( _{0\; 0}^{0\; 1} \big)}
\node{A_{n-1}\coprod A_{n-2}}\arrow{s,r}{(f_{n-1}\; -1)}\arrow{e,..,-}
\node{A_2\coprod A_1}\arrow{e,t}{\big( _{0\; 0}^{0\; 1}
\big)}\arrow{s,r}{(f_2\; -1)} \node{A_1\coprod A_{0}}\arrow{s,r}{(-f_1\;
1)}\\ \node{A_n}\arrow{e,t}{f_n} \node{A_{n-1}}\arrow{e,t}{f_{n-1}}
\node{A_{n-2}}\arrow{e,..,-} \node{A_1}\arrow{e,t}{f_1} \node{A_0}
\end{diagram}
\end{equation*}

Set $\mathbf{Q}_{n}=0\rightarrow (-,A_{n})\rightarrow 0$, and for $n-1\geq
i\geq 1$ consider the following complex $\mathbf{Q}_{i}$:
\begin{eqnarray*}
0\rightarrow (-,A_{n})\rightarrow (-,A_{n})\coprod (-,A_{n-1})\rightarrow\cdots\rightarrow\\
\rightarrow (-,A_{i+2})\coprod (-,A_{i+1})\rightarrow (-,A_{i+1})\coprod (-,A_{i})\rightarrow 0
\end{eqnarray*}

Then we have a relative projective resolution
\begin{equation*}
0\rightarrow \mathbf{Q}_{n}\rightarrow \mathbf{Q}_{n-1}\rightarrow \cdots
\rightarrow \mathbf{Q}_{1}\rightarrow \mathbf{Q}_{0}\rightarrow
P.\rightarrow 0
\end{equation*}%
with relative syzygy the complex:
\begin{equation*}
\Omega ^{i}(P.):0\rightarrow (-,A_{n})\rightarrow
(-,A_{n-1})\rightarrow (-,A_{n-2})\cdots (-,A_{i+2})\rightarrow (-,A_{i+1})\rightarrow 0
\end{equation*}%
for $n-1\geq i\geq 0$.

Therefore: we have an exact sequence
\[
 0\rightarrow \varTheta(\Omega
(P.))\rightarrow \varTheta(\mathbf{Q}_{0})\rightarrow \varTheta%
(P.)\rightarrow 0
\]
in $\mathrm{mod}(\mathcal{C})$. Since $\varTheta(\mathbf{Q%
}_{i})=(-,A_{i})$ and $\varTheta(P.)=M$, we have $\Omega (\varTheta(P.))=%
\varTheta(\Omega (P.))$, and we can prove by induction that $\Omega ^{i}(%
\varTheta(P.))=\varTheta(\Omega ^{i}P.)$, for all $i\geq 0$. It follows $%
\mathrm{rpdim}P.\geq \mathrm{pdim}\varTheta(P.).$

Conversely, applying $\varTheta$ to a relative projective resolution

\begin{equation*}
0\rightarrow \mathbf{Q}_{n}\rightarrow \mathbf{Q}_{n-1}\rightarrow \cdots
\rightarrow \mathbf{Q}_{1}\rightarrow \mathbf{Q}_{0}\rightarrow
P.\rightarrow 0,
\end{equation*}

we obtain a projective resolution of $\varTheta(P.)$

\begin{equation*}
0\rightarrow \varTheta(\mathbf{Q}_{n})\rightarrow \varTheta(\mathbf{Q}%
_{n-1})\rightarrow \cdots \rightarrow \varTheta(\mathbf{Q}_{1})\rightarrow %
\varTheta(\mathbf{Q}_{0})\rightarrow \varTheta (P.)\rightarrow 0.
\end{equation*}

It follows $\mathrm{rpdim}P.\leq \mathrm{pdim}\varTheta(P.).$
\end{proof}

As a corollary we have:

\begin{corollary}
Let $\mathcal{C}$ a dualizing Krull-Schmidt variety. If $P.$ and $Q$ are  are complexes
in $P^{0}(\mathrm{mod}(\mathcal{C}))$ without summands of the form
(\ref{relativeproj}), then there is an isomorphism
\begin{equation*}
\mathrm{Ext}_{C^{-}(\mathrm{mod}(\mathcal{C}))}^{k}(P.,Q.)=\mathrm{Ext}_{%
\mathrm{mod}(\mathcal{C})}^{k}(\varTheta(P.),\varTheta(Q.))
\end{equation*}
\end{corollary}

\begin{proof}
By Proposition \ref{omegaP}, we see that $\varTheta(\Omega ^{i}P.)=\Omega
^{i}(\varTheta(P.))$, $i\geq 0$. It is enough to prove the corollary for $%
k=1 $. Assume that (*) $0\rightarrow Q.\xrightarrow {(-,j_i)}E.%
\xrightarrow{(-,p_i)}P.\rightarrow 0$ is a exact sequence in $\mathrm{Ext}%
_{C^{-}(\mathrm{mod}(\mathcal{C})}^{k}(P.,Q.)$, with $Q.=\cdots \rightarrow
(-,B_{2})\xrightarrow{(-,g_2)}(-,B_{1})\xrightarrow{(-,g_1)}%
(-,B_{0})\rightarrow 0$, $P.=\cdots \rightarrow (-,A_{2})%
\xrightarrow{(-,f_2)}(-,A_{1})\xrightarrow{(-,f_1)}(-,A_{0})\rightarrow 0$, $%
E.=\cdots \rightarrow (-,E_{2})\xrightarrow{(-,h_2)}(-,E_{1})%
\xrightarrow{(-,h_1)}(-,E_{0})\rightarrow 0$. Since the exact sequence $0\rightarrow (-,B_{i})%
\xrightarrow {(-,j_i)}(-,E_{i})\xrightarrow{(-,p_i)}(-,A_{i})\rightarrow 0$
splits $E_{i}=A_{i}\coprod B_{i}$, $i\geq 0$. Then we have an exact sequence
in $\mathrm{mod}(\mathcal{C})$
\begin{equation}
0\rightarrow \varTheta(Q.)\xrightarrow{\rho}\varTheta(E.)\xrightarrow{\sigma}%
\varTheta(P.)\rightarrow 0  \label{splits}
\end{equation}%
If (\ref{splits}) splits, then there exist a map $\delta :\varTheta%
(E.)\rightarrow \varTheta(Q.)$ such that $\delta \rho =1_{\varTheta(Q.)}$,
We have a lifting of $\delta $, ${(-,l_{i})}_{i\in \mathbb{Z}}:E.\rightarrow
Q.$ such that the following diagram is commutative
\begin{equation*}
\begin{diagram} \node{\cdots}\arrow{e}
\node{(-,B_1)}\arrow{e,t}{(-,g_1)}\arrow{s,l}{(-,l_1j_i)}
\node{(-,B_0)}\arrow{e,t}{\pi}\arrow{s,l}{(-,l_0j_0)}
\node{\varTheta(Q.)}\arrow{e}\arrow{s,-,=} \node{0}\\ \node{\cdots}\arrow{e}
\node{(-,B_1)}\arrow{e,t}{(-,g_1)} \node{(-,B_0)}\arrow{e,t}{\pi}
\node{\varTheta(Q.)}\arrow{e} \node{0} \end{diagram}
\end{equation*}%
The complex $Q.$ has not summand of the form (\ref{relativeproj}), hence, $%
Q.$ is a minimal projective resolution of $\varTheta(Q.)$.

Since $\pi :(-,B_{0})\rightarrow \varTheta(Q.)$ is a projective cover, the
map $(-,l_{0}j_{0}):(-,B_{0})\rightarrow (-,B_{0})$ is an isomorphism, and
it follows by induction that all maps $(-,l_{i}j_{i})$ are isomorphisms,
which implies that the map $\{(-,j_{i})\}_{i\in\mathbb{Z}}:Q.\rightarrow E.$ is a splitting
homomorphism of complexes.

Given an exact sequence (**) $0\rightarrow G\rightarrow H\rightarrow
F\rightarrow 0$, in $\mathrm{mod}(\mathcal{C})$, we take minimal projective
resolutions $P.$ and $Q.$ of $F$ and $G,$ respectively, by the Horseshoe's Lemma,
 we have a projective resolution $E.$ for $H$, with $%
E_{i}=Q_{i}\oplus P_{i}$, and $0\rightarrow \varTheta(Q.)\rightarrow %
\varTheta(E.)\rightarrow \varTheta(P.)\rightarrow 0$ is a exact sequence in $%
\mathrm{mod}(\mathcal{C})$ isomorphic to (**).
\end{proof}

\subsection{Relative Tilting in $\mathrm{maps}(\mathcal{C})$}

Let $\mathcal{C}$ a dualizing Krull-Schmidt variety. In order to define an
exact structure on $\mathrm{maps}(\mathcal{C})$ we proceed as follows: we
identify first $\mathcal{C}$ with the category $\mathfrak{p(}\mathcal{C})$
of projective objects of $\mathcal{A}=\mathrm{mod}(\mathcal{C})$, in this
way $\mathrm{maps}(\mathcal{C})$ is equivalent to $\mathrm{maps}(\mathfrak{p(%
}\mathcal{C}))$ which is embedded in the abelian category $\mathcal{B}=$ $%
\mathrm{maps}(\mathcal{A}).$ We can define an exact structure ($\mathrm{maps}%
(\mathcal{C}),\mathcal{S})$ giving a subfunctor $F$of $\mathrm{Ext}_{\mathcal{B}%
}^{1}(-,?)$. Let $\varPsi:P^{0}(\mathcal{A})\rightarrow \mathrm{maps}(%
\mathcal{C})$ be the functor given above and $\alpha :$ $\mathrm{maps}(%
\mathcal{C})$ $\rightarrow $ $\mathrm{maps}(\mathfrak{p(}\mathcal{C}))$ the
natural equivalence. Since $\varPsi$ is dense any object in $\mathrm{maps}(%
\mathcal{C})$ is of the form $\varPsi(P.)$ and we define $\mathrm{Ext}_{F}^{1}$($\alpha %
\varPsi(P.)$ , $\alpha \varPsi(Q.))$ as $\alpha \varPsi(\mathrm{Ext}_{C^{-}(%
\mathcal{A})}^{1}(P.,Q.))$. We obtain the exact structure on $\mathrm{maps}(%
\mathcal{C})$ using the identification $\alpha .$

\bigskip Once we have the exact structure on $\mathrm{maps}(\mathcal{C})$
the definition of a relative tilting subcategory $\mathcal{T}_{\mathcal{C}}$
of $\mathrm{maps}(\mathcal{C})$ is very natural, it will be equivalent to
the following:

\begin{definition}
A relative tilting category in the category of maps, $\mathrm{maps}(\mathcal{%
C})$, is a subcategory $\mathcal{T}_{\mathcal{C}}$ such that :

\begin{itemize}
\item[(i)] Given $T:T_{1}\rightarrow T_{0}$ in $\mathcal{T}_{\mathcal{C}}$ ,
and $P.\in P^{0}(\mathcal{C})$ such that $\varPsi(P.)=T$, there exist an
integer $n$ such that $\mathrm{rpdim}P.\leq n$.

\item[(ii)] Given $T:T_{1}\rightarrow T_{0}$, $T^{\prime }:T_{1}^{\prime
}\rightarrow T_{0}^{\prime }$ in $\mathcal{T}_{\mathcal{C}}$ and $\varPsi%
(P.)=T$, $\varPsi(Q.)=T^{\prime }$, $P.,Q.\in P^{0}(\mathrm{mod}(\mathcal{C}%
))$. Then $\mathrm{Ext}_{C^{-}(\mathrm{mod}(\mathcal{C}))}^{k}(P.,Q.)=0$ for
all $k\geq 1$.

\item[(iii)] Given an object $C$ in $\mathcal{C}$, denote by $(-,C)_{\circ }$
the complex $0\rightarrow (-,C)\rightarrow 0$ concentrated in degree zero.
Then there exists an exact sequence
\begin{equation*}
0\rightarrow (-,C)_{\circ }\rightarrow P_{0}\rightarrow P_{1}\rightarrow
\cdots \rightarrow P_{n}\rightarrow 0
\end{equation*}%
with $P_{i}\in P^{0}(\mathrm{mod}(\mathcal{C}))$ and $\varPsi(P_{i})\in
\mathcal{T}_{\mathcal{C}}$.
\end{itemize}
\end{definition}

By definition, the following is clear

\begin{theorem}
\label{maps2} Let $\varPhi:\mathrm{maps}(\mathcal{C})\rightarrow \mathrm{mod}%
(\mathcal{C})$ be functor above, $\mathcal{T}_{\mathcal{C}}$ is a relative
tilting subcategory of $\mathrm{maps}(\mathcal{C})$ if and only if $\varPhi(%
\mathcal{T}_{\mathcal{C}})$ is a tilting subcategory of $\mathrm{mod}(%
\mathcal{C})$
\end{theorem}

\section{The Algebra of Triangular Matrices}

Let $\Lambda$ be an artin algebra. We want to explore the connections
between $\mathrm{mod}\ \Gamma $, with $\Gamma =\left(
\begin{array}{cc}
\Lambda & 0 \\
\Lambda & \Lambda%
\end{array}%
\right) $ and the category $\mathrm{mod}(\mathrm{mod}\Lambda )$. In
particular we want to compare the Auslander-Reiten quivers and subcategories
which are tilting, contravariantly, covariantly and functorially finite. We
identify $\mathrm{mod}\;\Gamma $ with the category of $\Lambda $-maps, $%
\mathrm{maps}(\Lambda )$ [see 7 Prop. 2.2]. We refer to the book by
Fossum, Griffits and Reiten [13] or to [16] for properties of modules over
triangular matrix rings.

\subsection{Almost Split Sequences}

In this subsection we want to study the relation between the almost split
sequences in $\mathrm{mod}\ \Gamma $ and almost split sequences in $\mathrm{%
mod}(\mathrm{mod}\Lambda ).$We will see that except for a few special
objects in $\mathrm{mod}\ \Gamma $, the almost split sequences will belong
to the class $\mathcal{S}$ of the exact structure, so in particular will be
relative almost split sequences.

For any indecomposable non projective $\Gamma $-module $M=(M_{1},M_{2},f)$
we can compute $DtrM$ as follows:

To construct a minimal projective resolution of $M$ ([13], [16]), let $P_{1}%
\xrightarrow{p_1}P_{0}\rightarrow M_{1}\rightarrow 0$ be a minimal
projective presentation. Taking the cokernel, we have an exact sequence $%
M_{1}\xrightarrow{f}M_{2}\xrightarrow{f_2}M_{3}\rightarrow 0$, and a
commutative diagram
\begin{equation*}
\begin{diagram}\dgARROWLENGTH=.5em \node{0}\arrow{e} \node{P_1}\arrow{e,t}{}
\arrow{s,l}{p_1} \node{P_1\oplus Q_1}\arrow{s,l}{}\arrow{e}
\node{Q_1}\arrow{e}\arrow{s,l}{q_1} \node{0}\\ \node{0}\arrow{e}
\node{P_0}\arrow{e,t}{} \arrow{s,l}{p_0} \node{P_0\oplus
Q_0}\arrow{s,l}{}\arrow{e} \node{Q_0}\arrow{e}\arrow{s,l}{q_0} \node{0}\\
\node{} \node{M_1}\arrow{e,t}{f}\arrow{s}
\node{M_2}\arrow{e,t}{f_2}\arrow{s} \node{M_3}\arrow{e}\arrow{s} \node{0}\\
\node{} \node{0} \node{0} \node{0} \end{diagram}
\end{equation*}%
with $Q_{0}$ the projective cover of $M_{3}.$The presentation can be written
as:

\begin{center}
$\bigskip $%
\begin{equation*}
\left(
\begin{array}{c}
P_{1} \\
P_{1}\oplus Q_{1}%
\end{array}%
\right) \rightarrow \left(
\begin{array}{c}
P_{0} \\
P_{0}\oplus Q_{0}%
\end{array}%
\right) \rightarrow \left(
\begin{array}{c}
M_{1} \\
M_{2}%
\end{array}%
\right) \rightarrow 0
\end{equation*}
\end{center}

and $trM$ will look as as follows:

\begin{equation*}
\left(
\begin{array}{c}
P_{0}^{\ast }\oplus Q_{0}^{\ast } \\
Q_{0}^{\ast }%
\end{array}%
\right) \rightarrow \left(
\begin{array}{c}
P_{1}^{\ast }\oplus Q_{1}^{\ast } \\
Q_{1}^{\ast }%
\end{array}%
\right) \rightarrow tr\left(
\begin{array}{c}
M_{1} \\
M_{2}%
\end{array}%
\right) \rightarrow 0
\end{equation*}

which corresponds to the commutative exact diagram:
\begin{equation*}
\begin{diagram}\dgARROWLENGTH=.5em \node{0}\arrow{e}
\node{Q_0^{\ast}}\arrow{e,t}{} \arrow{s,l}{q_1^{\ast}}
\node{Q_0^{\ast}\oplus P_0^{\ast}}\arrow{s,l}{}\arrow{e}
\node{P_0^{\ast}}\arrow{e}\arrow{s,l}{p_1^{\ast}} \node{0}\\
\node{0}\arrow{e} \node{Q_1^{\ast}}\arrow{e,t}{} \arrow{s,l}{}
\node{Q_1^{\ast}\oplus P_1^{\ast}}\arrow{s,l}{}\arrow{e}
\node{P_1^{\ast}}\arrow{e}\arrow{s} \node{0}\\ \node{} \node{tr M_3\oplus
Q^{\ast}}\arrow{e,t}{}\arrow{s} \node{tr M_2\oplus
P^{\ast}}\arrow{e,t}{}\arrow{s} \node{tr M_1}\arrow{e}\arrow{s} \node{0}\\
\node{} \node{0} \node{0} \node{0} \end{diagram}
\end{equation*}

\bigskip with $Q^{\ast }$, $P^{\ast }$, projectives coming from the fact
that the presentations of $M_{2}$ and $M_{3}$ in the first diagram are not
necessary minimal.

Then $\tau M$ is obtained as $\tau (M_{1},M_{2},f)=\tau M_{2}\oplus
D(P^{\ast })\rightarrow \tau M_{3}\oplus D(Q^{\ast })$, with kernel $%
0\rightarrow \tau M\rightarrow \tau M_{2}\oplus D(P^{\ast })\rightarrow \tau
M_{3}\oplus D(Q^{\ast })$.

We consider now the special cases of indecomposable $\Gamma $-modules of the
form: $M\xrightarrow{1_M}M$, $(M,0,0)$, $(0,M,0)$, with $M$ a non projective
indecomposable $\Lambda $-module.

\begin{proposition}
Let  $0\rightarrow \tau M\xrightarrow{j}E\xrightarrow{\pi}M\rightarrow 0$  be
an  almost   split sequence of $\Lambda $-modules.

\begin{itemize}
\item[(a)] Then the exact sequences of $\Gamma $-modules:

\begin{itemize}
\item[(1)] $0\rightarrow (\tau M,0,0)\xrightarrow{(j\ 0)}(E,M,\pi )%
\xrightarrow{(\pi\ 0)}(M,M,1_{M})\rightarrow 0$,

\item[(2)] $0\rightarrow (\tau M,\tau M,1_{\tau M})%
\xrightarrow {(1_{\tau
M}\ j)}(\tau M,E,j)\xrightarrow {(0\ \pi)}(0,M,0)\rightarrow 0$,
\end{itemize}

are almost split.

\item[(b)] Given a minimal projective resolution $P_{1}\xrightarrow{p_1}P_{0}%
\xrightarrow {p_0}M\rightarrow 0$, we obtain a commutative diagram:
\begin{equation}
\begin{diagram}\dgARROWLENGTH=1.5em \node{0}\arrow{e} \node{\tau
M}\arrow{e,t}{j}\arrow{s,=,-}
\node{E}\arrow{e,t}{\pi}\arrow{s,l}{\overline{t}}
\node{M}\arrow{e}\arrow{s,l}{t} \node{0}\\ \node{0}\arrow{e} \node{\tau
M}\arrow{e,t}{u} \node{D(P_1^{\ast})}\arrow{e,t}{D(p_1^{\ast})}
\node{D(P_0^{\ast})}\arrow{e,!} \node{} \end{diagram}  \label{almostspecial}
\end{equation}

Then the exact sequence
\[
0\rightarrow (N_{1},N_{2},g)\xrightarrow {(j_2\ j_1)}(E_{1},E_{2},h)%
\xrightarrow{(\pi_1\ \pi_2)}(M_{1},M_{2},f)\rightarrow 0\text{,}
\]
with  $(N_{1},N_{2},g)=(D(P_{1}^{\ast }),D(P_{0}^{\ast }),D(p_{1}^{\ast }))$,
 \ \ \ $(M_{1},M_{2},f)=(M,0,0)$, $(E_{1},E_{2},h)=(D(P_{1}^{\ast })\oplus M,D(P_{0}^{\ast
}),(D(p_{1}^{\ast })\ t))$  and \ \ $(j_2\ j_1)=(\left( ^1_0\right)\ 1)$,
$(\pi_1\ \pi_2)=(0\ -1)$,
is an almost split sequence.
\end{itemize}
\end{proposition}

\begin{proof}
(a) (1) Since $\pi :E\rightarrow M$ does not splits, the map $(\pi
,1_{M}):(E,M,\pi )\rightarrow (M,M,1_{M})$ does not split. Let $%
(q_{1},q_{2}):(X_{1},X_{2},f)\rightarrow (M,M,1_{M})$ be a map that is not a
splittable epimorphism. Then $q_{2}f=q_{1}1_{M}=q_{1}$.

We claim $q_{1}$ is not a splittable epimorphism. Indeed, if $q_{1}$ is a
splittable epimorphism, then there exists a morphism $s:M\rightarrow X_{1}$,
such that $q_{1}s=1_{M}$ and we have the following commutative diagram:
\begin{equation*}
\begin{diagram}\dgARROWLENGTH=1em \node{M}\arrow{s,l}{s}\arrow{e,t}{1_M}
\node{M}\arrow{s,l}{fs}\\ \node{X_1}\arrow{s,l}{q_1}\arrow{e,t}{f}
\node{X_2}\arrow{s,l}{q_2}\\ \node{M}\arrow{e,t}{1_M} \node{M} \end{diagram}
\end{equation*}%
with $q_{2}fs=q_{1}s=1_{M}$, and $(q_{1},q_{2}):(X_{1},X_{2},f)\rightarrow
(M,M,1_{M})$ is a splittable epimorphism, a contradiction.

Since $\pi :E\rightarrow M$ is a right almost split morphism, there exists a
map $h:X_{1}\rightarrow E$ such that $\pi h=q_{1}$, and $q_{2}f=q_{1}=\pi h$
. We have the following commutative diagram:
\begin{equation*}
\begin{diagram}\dgARROWLENGTH=1em \node{X_1}\arrow{s,l}{h}\arrow{e,t}{f}
\node{X_2}\arrow{s,l}{q_2}\\ \node{E}\arrow{s,l}{\pi}\arrow{e,t}{\pi}
\node{M}\arrow{s,l}{1_M}\\ \node{M}\arrow{e,t}{1_M} \node{M} \end{diagram}
\end{equation*}%
with $(\pi \ 1_{M})(h\ q_{2})=(q_{1}\ q_{2}).$ We get a lifting $(h\
q_{2}):(X_{1},X_{2},f)\rightarrow (E,M,\pi )$ of $(q_{1},q_{2})$. We have
proved $\tau (M,M,1_{M})=(\tau M,0,0)$.

(2) It is clear, $\tau (0,M,0)=(\tau M,\tau M,1_{\tau M})$ and the exact
sequence does not split. Now, let $(0,\rho ):(0,M,0)\rightarrow (0,M,0)$ be
a non isomorphism. Then there exists a map $h:M\rightarrow E$ with $\pi
h=\rho $. We have $(0\ \pi )(0\ h)=(0\ \rho )$.

(b) We have the following commutative diagram:

\begin{equation*}
\begin{diagram}\dgARROWLENGTH=1em
\node{P_1}\arrow{s,r}{p_1}\arrow{e,t}{\left( ^{1_{P_1}}_0\right)}
\node{P_1\oplus P_0}\arrow{s,r}{(p_1\; 1_{P_0})}\arrow{e,t}{(p_1\; 1_{P_0})}
\node{P_0}\\ \node{P_0}\arrow{s,l}{p_0}\arrow{e,t}{1_{P_0}}
\node{P_0}\arrow{s}\\ \node{M}\arrow{e,t}{}\arrow{s} \node{0}\\ \node{0}
\node{} \end{diagram}
\end{equation*}%
which implies the existence of the following commutative diagram:
\begin{equation*}
\begin{diagram}\dgARROWLENGTH=1.5em \node{0}\arrow{s,r}{}\arrow{e,t}{}
\node{P_0^{\ast}}\arrow{s,r}{\left(
^{1_{P_0^{\ast}}}_{p_1^{\ast}}\right)}\arrow{e,t}{1_{P_0^{\ast}}}
\node{P_0^{\ast}}\arrow{s,r}{p_1^{\ast}}\\
\node{P_0^{\ast}}\arrow{s,r}{1_{P_0^{\ast}} }\arrow{e,t}{\left(
^{1_{P_0^{\ast}}}_{p_1^{\ast}}\right)} \node{P_0^{\ast}\oplus
P_1^{\ast}}\arrow{e,t}{(0\;1_{P_1^{\ast}})}\arrow{s,r}{(p_1^{\ast}%
\;-1_{P_1^{\ast}})} \node{P_1^{\ast}}\arrow{s}\\
\node{P_0^{\ast}}\arrow{e,t}{p_1^{\ast}}\arrow{s}
\node{P_1^{\ast}}\arrow{e,t}{}\arrow{s} \node{tr M}\arrow{s}\\ \node{0}
\node{0} \node{0} \end{diagram}
\end{equation*}%
and $Dtr(M,0,0)=(D(P_{1}^{\ast })\xrightarrow{D(p_1^{\ast})}D(P_{0}^{\ast
})) $ Since $j:\tau M\rightarrow E$ is a left almost split map, it extends
to the map $\tau M\rightarrow D(P_{1}^{\ast })$. We have the commutative
diagram (\ref{almostspecial}).

Hence, $E$ is the pullback of the maps $t:M\rightarrow D(P_0^{\ast})$, $%
D(P_1^{\ast})\rightarrow D(P_0^{\ast})$. We have an exact sequence:
\begin{equation*}
0\rightarrow E\xrightarrow{\left(^{\overline{t}}_{-\pi}\right)}%
D(P_1^{\ast})\oplus M\xrightarrow {(D(p_1^{\ast})\;t)} D(P_0^{\ast})
\end{equation*}

from which we built an exact commutative diagram:
\begin{equation*}
\begin{diagram}\dgARROWLENGTH=1.5em \node{\tau
M}\arrow{s,r}{j}\arrow{e,t}{u} \node{D(P_1^{\ast})}\arrow{s,l}{\left(
^{1_{D(P_1^{\ast})}}_0\right)}\arrow{e,t}{D(p_1^{\ast})}
\node{D(P_0^{\ast})}\arrow{s,r}{1_{D(P_0^{\ast})}}\\
\node{E}\arrow{s,l}{\pi}\arrow{e,t}{\left(^{\overline{t}}_{-\pi}\right)}
\node{D(P_1^{\ast})\oplus M}\arrow{s,l}{(0\;
-1_{M})}\arrow{e,t}{(D(p_1^{\ast})\; t)} \node{D(P_0^{\ast})}\arrow{s}\\
\node{M}\arrow{e,t}{1_{M}}\arrow{s} \node{M}\arrow{e,t}{}\arrow{s}
\node{0}\\ \node{0} \node{0} \node{} \end{diagram}
\end{equation*}

We claim that the exact sequence
\begin{equation*}
0\rightarrow (D(P_{1}^{\ast }),D(P_{0}^{\ast }),D(p_{1}^{\ast }))\rightarrow
(D(P_{1}^{\ast })\oplus M,D(P_{0}^{\ast }),(D(p_{1}^{\ast })\;t))\rightarrow
(M,0,0)\rightarrow 0
\end{equation*}%
is an almost split sequence.

We need to prove first that it does not split. Suppose there exists a map $%
\left( _{v}^{\mu }\right) :(M,0,0)\rightarrow (D(P_{1}^{\ast })\oplus
M,D(P_{0}^{\ast }),(D(p_{1}^{\ast })\;t))$ such that $((0\;-1_{M})\;0)(%
\left( _{v}^{\mu }\right) \;0)=(1_{M}\;0)$, then $(D(p_{1}^{\ast
})\;t)\left( _{v}^{\mu }\right) =0$ and $v=-1_{M}$. It follows that there
exists $s:M\rightarrow E$ such that $\left( _{-\pi }^{\overline{t}}\right)
s=\left( _{v}^{\mu }\right) $, therefore $-\pi s=v=-1$, which implies $%
0\rightarrow \tau M\rightarrow E\xrightarrow{\pi}M\rightarrow 0$ splits.

It will be enough to prove that any automorphism $(\sigma
,0):(M,0,0)\rightarrow (M,0,0)$, which is not an isomorphism, lifts to $%
D(P_{1}^{\ast })\oplus M\rightarrow D(P_{0}^{\ast })$. But there exist a map
$s:M\rightarrow E$ with $\pi s=\sigma $. We have a map
\begin{equation*}
(\left( _{-\pi s}^{\overline{t}s}\right) \ 0):(M,0,0)\rightarrow
(D(P_{1}^{\ast })\oplus M,D(P_{0}^{\ast }),(D(p_{1}^{\ast })\;t))
\end{equation*}%
such that $((0\;-1_{M})\;0)(\left( _{-\pi s}^{\overline{t}s}\right) \
0)=(\sigma \;0)$.
\end{proof}

Dually, we consider almost split sequences of the form
\begin{equation*}
0\rightarrow (N_{1},N_{2},g)\xrightarrow {(j_2\ j_1)}(E_{1},E_{2},h)%
\xrightarrow{(\pi_1\ \pi_2)}(M_{1},M_{2},f)\rightarrow 0\text{,}
\end{equation*}%
such that $(N_{1},N_{2},g)$ is one of the following cases $%
(N,N,1_{N}),(N,0,0),(0,N,0)$ , with $N$ a non injective indecomposable $%
\Lambda $-module to have the following:

\begin{proposition}
Let $0\rightarrow N\xrightarrow{j}E\xrightarrow {\pi}\tau^{-1}N\rightarrow 0$
an almost split sequence of $\Lambda$-modules.

\begin{itemize}
\item[(a)] Then the exact sequences of $\Gamma$-modules

\begin{itemize}
\item[(1)] $0\rightarrow (N,N,1_N)\xrightarrow{(1_N\ j)}(\tau M,E,j)%
\xrightarrow{(0\ \pi)}(0,M,0)\rightarrow 0$

\item[(2)] $0\rightarrow (N,0,0)\xrightarrow {(j\ 0)} (E,\tau^{-1}N,\pi)%
\xrightarrow {(\pi \ 1)} (\tau^{-1}N,\tau^{-1}N,1_{\tau^{-1}N})\rightarrow 0$
\end{itemize}

are almost split.

\item[(b)] Given a minimal injective resolution  $0\rightarrow N%
\xrightarrow{q_0}I_{0}\xrightarrow{q_1}I_{1}$ , we obtain a commutative
diagram
\begin{equation*}
\begin{diagram}\dgARROWLENGTH=1.5em \node{}\arrow{e,!}
\node{D(I_0)^{\ast}}\arrow{e,t}{D(q_1)^{\ast}}\arrow{s,l}{v}
\node{D(I_1)^{\ast}}\arrow{e,t}{}\arrow{s,l}{\overline{v}}
\node{\tau^{-1}N}\arrow{e}\arrow{s,=}{} \node{0}\\ \node{0}\arrow{e}
\node{N}\arrow{e,t}{j} \node{E}\arrow{e,t}{\pi} \node{\tau^{-1}N}\arrow{e}
\node{0} \end{diagram}
\end{equation*}%
Then the exact sequence
\[
0\rightarrow (N_{1},N_{2},g)\xrightarrow {(j_2\ j_1)}(E_{1},E_{2},h)%
\xrightarrow{(\pi_1\ \pi_2)}(M_{1},M_{2},f)\rightarrow 0\text{,}
\]
with $(N_{1},N_{2},g)=(0,N,0)$, $(E_{1},E_{2},h)=(D(I_{0})^{\ast
},D(I_{1})^{\ast }\oplus N,\left( _{v}^{D(q_{1})^{\ast }}\right) )$,
$(M_{1},M_{2},f)=(D(I_{0})^{\ast },D(I_{0})^{\ast },D(q_{1})^{\ast
})$ and $(j_2\ j_1)=(0\ \left( ^0_1\right))$, $(\pi_1\ \pi_2)=(1\ (1\ 0))$,
is an almost split sequence.
\end{itemize}
\end{proposition}

We will prove next that almost split sequences of objects which do not
belong to the special cases consider before, are exact sequences in the
relative structure $\mathcal{S}.$

\begin{theorem}
Let
\begin{equation*}
0\rightarrow (N_{1},N_{2},g)\xrightarrow{(j_1\ j_2)}(E_{1},E_{2},h)%
\xrightarrow{(p_1\ p_2)}(M_{1},M_{2},f)\rightarrow 0
\end{equation*}%
be an almost split sequence of $\Gamma $-modules and assume that both $g,f,$
are neither splittable epimorphisms, nor splittable monomorphisms. Consider
the following commutative commutative exact diagram:
\begin{equation}
\begin{diagram}\dgARROWLENGTH=1em \node{} \node{0}\arrow{s}
\node{0}\arrow{s} \node{0}\arrow{s} \node{} \node{}\\ \node{0}\arrow{e,t}{}
\node{N_0}\arrow{s,l}{j_0}\arrow{e,t}{}
\node{N_1}\arrow{s,l}{j_1}\arrow{e,t}{}
\node{N_2}\arrow{s,l}{j_2}\arrow{e,t}{}
\node{N_3}\arrow{s,l}{j_3}\arrow{e,t}{} \node{0}\\ \node{0}\arrow{e,t}{}
\node{E_0}\arrow{s,l}{p_0}\arrow{e,t}{}
\node{E_1}\arrow{s,l}{p_1}\arrow{e,t}{}
\node{E_2}\arrow{s,l}{p_2}\arrow{e,t}{}
\node{E_3}\arrow{s,l}{p_3}\arrow{e,t}{} \node{0}\\ \node{0}\arrow{e,t}{}
\node{M_0}\arrow{e,t}{} \node{M_1}\arrow{s}\arrow{e,t}{}
\node{M_2}\arrow{s}\arrow{e,t}{} \node{M_3}\arrow{s}\arrow{e,t}{} \node{0}\\
\node{} \node{} \node{0} \node{0} \node{0} \node{} \end{diagram}
\label{diagramalmost}
\end{equation}

Then, the map $p_0:E_0\rightarrow M_0$ is an epimorphism, $%
j_3:N_3\rightarrow E_3$ is a monomorphism, and the exact sequences
\begin{equation}
0\rightarrow N_i\xrightarrow{j_i} E_i\xrightarrow{p_i} M_i\rightarrow 0,\
0\le i\le 3
\end{equation}
split.
\end{theorem}

\begin{proof}
The map $(1_{M_{1}}\ f):(M_{1},M_{1},1_{M_{1}})\rightarrow (M_{1},M_{2},f)$
is not a splittable epimorphism. Therefore it factors through $E_{1}%
\xrightarrow{h}E_{2}$.  Hence; there exists a map $(t_{1}\
t_{2}):(M_{1},M_{2},f)\rightarrow (E_{1},E_{2},h)$, with $(p_{1}\
p_{2})(t_{1}\ t_{2})=(1_{M_{1}}\ f)$. Similarly, the map $(0\
1_{M_{2}}):(0,M_{2},0)\rightarrow (M_{1},M_{2},f)$ is not a splittable
epimorphism. Hence; there exists a map $(t_{1}\
t_{2}):(M_{1},M_{2},f)\rightarrow (E_{1},E_{2},h)$ with $(p_{1}\
p_{2})(t_{1}\ t_{2})=(0\ 1_{M_{2}})$. We have proved that  for $i=1,2,$ the
exact sequences $0\rightarrow N_{i}\xrightarrow{j_i}E_{i}\xrightarrow{p_i}%
M_{i}\rightarrow 0$ split.

The diagram (\ref{diagramalmost}) induces the following commutative diagram
\begin{equation*}
\begin{diagram}\dgARROWLENGTH=1em \node{} \node{0}\arrow{s}
\node{0}\arrow{s} \node{0}\arrow{s} \node{} \node{}\\ \node{0}\arrow{e,t}{}
\node{(-,N_0)}\arrow{s,l}{(-,j_0)}\arrow{e,t}{}
\node{(-,N_1)}\arrow{s,l}{(-,j_1)}\arrow{e,t}{(-,g)}
\node{(-,N_2)}\arrow{s,l}{(-,j_2)}\arrow{e,t}{\tau}
\node{G}\arrow{s,l}{\rho}\arrow{e,t}{} \node{0}\\ \node{0}\arrow{e,t}{}
\node{(-,E_0)}\arrow{s,l}{(-,p_0)}\arrow{e,t}{}
\node{(-,E_1)}\arrow{s,l}{(-,p_1)}\arrow{e,t}{(-,h)}
\node{(-,E_2)}\arrow{s,l}{(-,p_2)}\arrow{e,t}{\pi}
\node{H}\arrow{s,l}{\sigma}\arrow{e,t}{} \node{0}\\ \node{0}\arrow{e,t}{}
\node{(-,M_0)}\arrow{e,t}{} \node{(-,M_1)}\arrow{s}\arrow{e,t}{(-,f)}
\node{(-,M_2)}\arrow{s}\arrow{e,t}{\eta} \node{F}\arrow{s}\arrow{e,t}{}
\node{0}\\ \node{} \node{} \node{0} \node{0} \node{0} \node{} \end{diagram}
\end{equation*}%
By the Snake's Lemma, we have a connecting map $\delta $,
\begin{equation*}
\cdots \rightarrow (-,E_{0})\xrightarrow{(-,p_0)}(-,M_{0})%
\xrightarrow{\delta}G\xrightarrow{\rho}H\rightarrow \cdots
\end{equation*}%
We want to prove $\rho $ is a monomorphism. Let $\rho :G\xrightarrow{\rho_1}%
\mathrm{Im}\rho \xrightarrow{\rho_2}H$ be a factorization through its image.

Since $\mathrm{mod}(\mathrm{mod}\Lambda )$ is an abelian category, $\mathrm{%
Im}\rho $ is a finitely presented functor, with presentation

\begin{equation*}
(-,X_{1})\xrightarrow{(-,t)}(-,X_{2})\rightarrow \mathrm{Im}\rho \rightarrow
0.
\end{equation*}

Lifting the maps $\rho _{1},\rho _{2}$ we obtain a commutative diagram with
exact rows:
\begin{equation*}
\begin{diagram}\dgARROWLENGTH=1em
\node{(-,N_1)}\arrow{e,t}{(-,g)}\arrow{s,l}{(-,u_1)}
\node{(-,N_2)}\arrow{e,t}{}\arrow{s,l}{(-,u_2)}
\node{G}\arrow{e,t}{}\arrow{s,l}{\rho_1} \node{0}\\
\node{(-,X_1)}\arrow{e,t}{(-,t)}\arrow{s,l}{(-,v_1)}
\node{(-,X_2)}\arrow{e,t}{}\arrow{s,l}{(-,v_2)}
\node{\mathrm{Im}\rho}\arrow{e,t}{}\arrow{s,l}{\rho_1} \node{0}\\
\node{(-,E_1)}\arrow{e,t}{(-,h)} \node{(-,E_2)}\arrow{e,t}{}
\node{H}\arrow{e,t}{} \node{0} \end{diagram}
\end{equation*}%
whose composition is another lifting of $\rho $. Then the two liftings are
homotopic and there exist maps $(-,s_{1}):(-,N_{2})\rightarrow (-,E_{1})$, $%
(-,s_{2}):(-,N_{1})\rightarrow (-,E_{0})$ such that $%
(-,j_{2})=(-,h)(-,s_{1})+(-,v_{1}u_{1})$, $%
(-,j_{2})=(-,h_{1})(-,s_{2})+(-,s_{1})(-,g)+(-,v_{2}u_{2})$. This is $%
j_{2}=hs_{1}+v_{1}u_{1}$, $j_{1}=h_{1}s_{2}+s_{1}g+v_{2}u_{2}$. Consider the
following commutative diagram
\begin{equation*}
\begin{diagram}\dgARROWLENGTH=2.5em
\node{N_1}\arrow{s,l}{\left(\begin{matrix} u_1\\ s_1g\\ s_2
\end{matrix}\right) }\arrow{e,t}{g}
\node{N_2}\arrow{s,r}{\left(\begin{matrix} u_2\\ s_1 \end{matrix}\right) }\\
\node{X_1\oplus E_1\oplus E_0}\arrow{e,t}{\left(\begin{matrix} t&0&0\\
0&1_{E_1}&0 \end{matrix}\right) }\arrow{s,l}{(v_1\ 1_{E_1}\ h_1)}
\node{X_2\oplus E_1}\arrow{s,r}{(v_2\ h)}\\ \node{E_1}\arrow{e,t}{h}
\node{E_2} \end{diagram}
\end{equation*}

But
\begin{equation}
(-,X_{1}\oplus E_{1}\oplus E_{0})\xrightarrow{\left(-,\left(\begin{matrix}
t&0&0\\ 0&1_{E_1}&0 \end{matrix}\right)\right)}(-,X_{2}\oplus
E_{1})\rightarrow \mathrm{Im}\rho \rightarrow 0  \label{diagramalmost1}
\end{equation}%
is exact. Changing $(-,X_{1})\xrightarrow{(-,t)}X_{2}$ by (\ref%
{diagramalmost1}), we can assume $v_{1}u_{1}=l_{i}$, $i=1,2$; but being $%
(l_{1},l_{2}):(N_{1},N_{2},g)\rightarrow (E_{1},E_{2},h)$ an irreducible
map, this implies either $(u_{1},u_{2}):(N_{1},N_{2},g)\rightarrow
(X_{1},X_{2},t)$ is a splittable monomorphism or $%
(v_{1},v_{2}):(X_{1},X_{2},t)\rightarrow (E_{1},E_{2},h)$ is a splittable
epimorphism.

In the second case we have a map $(s_{1},s_{2}):(E_{1},E_{2},h)\rightarrow
(X_{1},X_{2},t),$ with $(v_{1}\ v_{2})(s_{1}\ s_{2})=(1_{E_{1}}\ 1_{E_{2}})$%
. Then there exists a map $\sigma :H\rightarrow \mathrm{Im}\rho $, such that
$\rho _{2}\sigma =1_{H}$. It follows $\rho _{2}$ is an isomorphism. Hence; $%
F=0$ and $f:M_{1}\rightarrow M_{2}$ is a splittable epimorphism. A
contradiction.

Now, if $(u_{1}\ u_{2})$ is a splittable monomorphism, then there exists a
map $(q_{1}\ q_{2}):(X_{1},X_{2},t)\rightarrow (N_{1},N_{2},g)$, with $%
(q_{1}\ q_{2})(u_{1}\ u_{2})=(1_{N_{1}}\ 1_{N_{2}})$. Then, there exists $%
\sigma :\mathrm{Im}\rho \rightarrow G$ such that $\sigma \rho =1_{G}$, and $%
\rho _{1}$ is an isomorphism, in particular $\rho $ is a monomorphism.It
follows $(-,E_{0})\xrightarrow{(-,p_0)}(-,M_{0})$ is an epimorphism.
Therefore: $0\rightarrow N_{0}\xrightarrow{j_0}E_{0}\xrightarrow{p_0}%
M_{0}\rightarrow 0$ is an exact sequence that splits. A contradiction.

Dualizing the diagram, we obtain, the exact sequence $0\rightarrow
D(M_{3})\rightarrow D(E_{3})\rightarrow D(N_{3})\rightarrow 0$ splits.
Therefore the exact sequence $0\rightarrow N_{3}\rightarrow E_{3}\rightarrow
M_{3}\rightarrow 0$ splits.
\end{proof}

We can see now that the functor $\Phi $ preserves almost split sequences.

\begin{theorem}
Let
\[
 0\rightarrow (N_{1},N_{2},g)\xrightarrow{(j_1\ j_2)}(E_{1},E_{2},h)%
\xrightarrow{(p_1\ p_2)}(M_{1},M_{2},f)\rightarrow 0
\]
be an almost split sequence, such that $g,f,$ are neither splittable epimorphisms nor
splittable monomorphisms. Then the exact sequence
\[
 0\rightarrow G%
\xrightarrow{\rho}H\xrightarrow{\theta}F\rightarrow 0
\]
 obtained from the commutative diagram:
\begin{equation*}
\begin{diagram}\dgARROWLENGTH=.7em \node{0}\arrow{s} \node{0}\arrow{s}
\node{0}\arrow{s} \node{}\\
\node{(-,N_1)}\arrow{s,l}{(-,j_1)}\arrow{e,t}{(-,g)}
\node{(-,N_2)}\arrow{s,l}{(-,j_2)}\arrow{e,t}{}
\node{G}\arrow{s,l}{\rho}\arrow{e,t}{} \node{0}\\
\node{(-,E_1)}\arrow{s,l}{(-,p_1)}\arrow{e,t}{(-,h)}
\node{(-,E_2)}\arrow{s,l}{(-,p_2)}\arrow{e,t}{}
\node{H}\arrow{s,l}{\theta}\arrow{e,t}{} \node{0}\\
\node{(-,M_1)}\arrow{s}\arrow{e,t}{(-,f)}
\node{(-,M_2)}\arrow{s}\arrow{e,t}{} \node{F}\arrow{s}\arrow{e,t}{}
\node{0}\\ \node{0} \node{0} \node{0} \node{} \end{diagram}
\end{equation*}%
is an almost split sequence
\end{theorem}

\begin{proof}

(1) The sequence $0\rightarrow G\xrightarrow{\rho}H\xrightarrow{\theta}%
E\rightarrow 0$ does not split.

Assume it does split and let $u:F\rightarrow H$, with $\theta u=1_{F}$ be
the splitting. There is a lifting of $u$ making the following diagram, with
exact raws, commute:
\begin{equation*}
\begin{diagram}\dgARROWLENGTH=.7em \node{(-,M_0)}
\arrow{e,t}{(-,f_1)}\arrow{s,l}{(-,s_0)} \node{(-,M_1)}
\arrow{e,t}{(-,f)}\arrow{s,l}{(-,s_1)}
\node{(-,M_2)}\arrow{e}\arrow{s,l}{(-,s_2)} \node{F}\arrow{e}\arrow{s}
\node{0}\\ \node{(-,E_0)} \arrow{e,t}{(-,h_1)}\arrow{s,l}{(-,p_0)}
\node{(-,E_1)} \arrow{e,t}{(-,h)}\arrow{s,l}{(-,p_1)}
\node{(-,E_2)}\arrow{e}\arrow{s,l}{(-,p_2)} \node{H}\arrow{e}\arrow{s}
\node{0}\\ \node{(-,M_0)} \arrow{e,t}{(-,f_1)} \node{(-,M_1)}
\arrow{e,t}{(-,f)} \node{(-,M_2)}\arrow{e} \node{F}\arrow{e} \node{0}
\end{diagram}
\end{equation*}%
The composition is a lifting of the identity, and as before, it is homotopic
to the identity. By Yoneda's lemma, there exist maps, $w_{2}:M_{2}%
\rightarrow M_{1}$, $w_{1}:M_{1}\rightarrow M_{0}$, such that $%
fw_{2}+p_{2}s_{2}=1_{M_{2}}$, $w_{2}f+f_{1}w_{1}+p_{1}s_{1}=1_{M_{1}}$.
Since $f\in \mathrm{rad}(M_{1},M_{2})$, $f_{1}\in \mathrm{rad}(M_{0},M_{1})$%
, this implies $w_{2}f,f_{1}w_{1}\in \mathrm{rad}\mathrm{End}(M_{1})$ and $%
fw_{2}\in \mathrm{rad}\mathrm{End}(M_{2})$. It follows $%
p_{2}s_{2}=1_{M_{2}}-fw_{2}$ and $p_{1}s_{1}=1_{M_{1}}-(w_{2}f+f_{1}w_{1})$
are invertible. Therefore: $(p_{1}\ p_{2}):(E_{1},E_{2},h)\rightarrow
(M_{1},M_{2},f)$ is a splittable epimorphism. A contradiction.

(2) Let $\eta :L\rightarrow F$ be a non splittable epimorphism, and $%
(-,X_{1})\xrightarrow{(-,t)}(-,X_{2})\rightarrow L\rightarrow 0$ a
projective presentation of $L$. The map $\eta $ lifts to a map $(-,\eta
_{i}):(-,X_{i})\rightarrow (-,M_{i})$, $i=1,2$. Then the map $(\eta _{1}\
\eta _{2}):(X_{1},X_{2},t)\rightarrow (M_{1},M_{2},f)$ is not a splittable
epimorphism., and there exists a map $(v_{1}\
v_{2}):(X_{1},X_{2},t)\rightarrow (E_{1},E_{2},h)$, with $(p_{1}\
p_{2})(v_{1}\ v_{2})=(\eta _{1}\ \eta _{2})$.

The map $(v_{1}\ v_{2})$ induces a map $\sigma :L\rightarrow H$ with $\theta
\sigma =\eta $.

In a similar way we prove $0\rightarrow G\rightarrow H$ is left almost split.
\end{proof}

Assume now (*) $0\rightarrow G\rightarrow H\rightarrow F\rightarrow 0$ is an
almost split sequence in $\mathrm{mod}(\mathrm{mod}\Lambda )$. Let $(-,M_{1})%
\xrightarrow{(-,f)}(-,M_{2})\rightarrow F\rightarrow 0$ be a minimal
projective projective presentation of $F$. The map $M_{1}\xrightarrow{f}%
M_{2} $ is an indecomposable object in $\mathrm{mod}\left(
\begin{matrix}
\Lambda & 0 \\
\Lambda & \Lambda%
\end{matrix}%
\right) $ and is not projective.

Then we have an almost split sequence in $\mathrm{maps}(\Lambda )$:
\begin{equation*}
0\rightarrow N=(N_{1},N_{2},g)\xrightarrow{(j_1\ j_2)}E=(E_{1},E_{2},h)%
\xrightarrow{(p_1\ p_2)}M=(M_{1},M_{2},f)\rightarrow 0
\end{equation*}%
where $f,g$ are both neither splittable monomorphisms  and nor splittable
epimorphisms. Applying the functor $\varPhi$ to the above sequence we
obtain an almost split sequence (**) $0\rightarrow \varPhi(N)\rightarrow %
\varPhi(E)\rightarrow \varPhi(M)\rightarrow 0$ with $\varPhi(M)=F$. By the
uniqueness of the almost split sequence $\varPhi(N)=G$, $\varPhi(E)=H$ and
(*) is isomorphic to (**).

\subsubsection{An example}

Let $\Lambda =\left(
\begin{matrix}
K & 0 \\
K & K%
\end{matrix}%
\right) $ be the algebra isomorphic to to the quiver algebra $KQ$, where $Q$
is: $1\rightarrow 2$ and $\Gamma =\left(
\begin{array}{cc}
\Lambda  & 0 \\
\Lambda  & \Lambda
\end{array}%
\right) $. The algebra $\Lambda $ has a simple projective $S_{2}$, a simple
injective $S_{1}$ and a projective injective $P_{1}$ .

The projective $\Gamma $-modules correspond to the maps: $0\rightarrow S_{2}$%
, $0\rightarrow P_{1}$, $P_{1}\xrightarrow{1_{P_1}}P_{1}$ and $S_{2}%
\xrightarrow{1_{S_2}}S_{2}$. We compute the almost split sequences in
\textrm{maps(mod(}$\Lambda ))$ to obtain exact sequences:
\begin{equation*}
0\rightarrow N\xrightarrow{j}E\xrightarrow{\pi}M\rightarrow 0
\end{equation*}%
with:

(a) $N=(0,S_{2},0)$, $E=(S_{2},S_{2}\oplus P_{1},\left(
_{0}^{1_{S_{2}}}\right) )$, $M=(S_{2},P_{1},f)$, $j=(0\ $,$\left(
_{f}^{1_{S_{2}}}\right) )$, $\pi =(1_{S_{2}}$,$(f\ $,$-1_{P_{1}}))$

(b) $N=(S_{2},P_{1},f)$, $E=(P_{1}\oplus S_{2},P_{1}\oplus S_{1},\left(
_{0}^{1_{P_{1}}}\ _{0}^{0}\right) )$, $M=(P_{1},S_{1},g)$, $j=(\left(
_{1_{S_{2}}}^{f}\right) $,$\left( _{g}^{1_{P_{1}}}\right) )$, $\pi
=((-1_{P_{1}},\ f),\ (-g\ ,1_{S_{1}})$.

(c) $N=(P_{1},S_{1},g)$, $E=(S_{1}\oplus P_{1},S_{1},(1_{S_{1}}\ 0))$, $%
M=(S_{1},0,0)$, $j=(\left( _{1_{P_{1}}}^{g}\right) \ 1_{S_{1}})$, $\pi
=((-1_{S_{1}}\ g)\ 0)$.

The Auslander-Reiten quiver in $\mathrm{maps}(\Lambda)$ is:
\[
\begin{diagram}\dgARROWLENGTH=.05em \dgTEXTARROWLENGTH=.5em
\dgHORIZPAD=.6em
\dgVERTPAD=.6ex
 \node{} \node{(S_2,S_2,1)}\arrow{se}
  \node{} \node{(P_1,P_1,1)}\arrow{se}
   \node{} \node{(S_1,S_1,1)}\arrow{se}
    \node{}\\
\node{(0,S_2,0)}\arrow{ne}\arrow{se}
 \node{}
  \node{(S_2,P_1,f)}\arrow{ne}\arrow{e}\arrow{se}
  \node{(S_2,0,0)}\arrow{e}
  \node{(P_1,S_1,g)}\arrow{ne}\arrow{se} \node{} \node{(S_1,0,0)}\\
 \node{}
  \node{(0,P_1,0)}\arrow{ne}
    \node{} \node{(0,S_1,0)}\arrow{ne}
      \node{}
           \node{(P_1,0,0)}\arrow{ne}
 \end{diagram}
\]

Applying the functor $\varPhi$ we obtain the following Auslander-Reiten
quiver of $\mathrm{mod}(\mathrm{mod}\Lambda ):$
\begin{equation*}
0\rightarrow (-,S_{2})\rightarrow (-,P_{1})\rightarrow \mathrm{rad}%
(-,S_{1})\rightarrow (-,S_{1})\rightarrow S_{S_{1}}\rightarrow 0,
\end{equation*}%
which is isomorphic to the Auslander-Reiten quiver of $1\xrightarrow{\alpha}2%
\xrightarrow{\beta}3$, with $\beta \alpha =0$, and this is the Auslander
algebra of $\Lambda $.

\subsection{Tilting in $\mathrm{mod}(\mathrm{mod}\varLambda)$}

Let $\Lambda $ be an artin algebra. Since $\mathrm{maps}(\mathrm{mod}\Lambda
)$ is equivalent to the category $\mathrm{mod}\ \Gamma $, with $\Gamma
=\left(
\begin{matrix}
\Lambda  & 0 \\
\Lambda  & \Lambda
\end{matrix}%
\right) $, it is abelian, dualizing Krull-Schmidt, and it has kernels. Hence
it has pseudokernels, and we can apply the theory so far developed. In this
case the exact structure is easy to describe.

The collection of exact sequences $\mathcal{S}$ consists of the exact
sequences in the category $\mathrm{maps}(\mathrm{mod}\Lambda )$
\begin{equation*}
0\rightarrow (N_{1},N_{2},g)\rightarrow (E_{1},E_{2},h)\rightarrow
(M_{1},M_{2},f)\rightarrow 0
\end{equation*}%
such that in the following exact commutative diagram
\begin{equation}
\begin{diagram}\dgARROWLENGTH=.5em \node{}\node{0}\arrow{s,l}{}
\node{0}\arrow{s,l}{} \node{0}\arrow{s,l}{}\\ \node{0}\arrow{e,t}{}
\node{N_0}\arrow{e,t}{g_0}\arrow{s,l}{}
\node{N_1}\arrow{e,t}{g}\arrow{s,l}{} \node{N_2}\arrow{s,l}{}\\
\node{0}\arrow{e,t}{} \node{E_0}\arrow{e,t}{h_0}\arrow{s,l}{}
\node{E_1}\arrow{e,t}{h}\arrow{s,l}{} \node{E_2}\arrow{s,l}{}\\
\node{0}\arrow{e,t}{} \node{M_0}\arrow{e,t}{f_0}\arrow{s,l}{}
\node{M_1}\arrow{e,t}{f}\arrow{s,l}{} \node{M_2}\arrow{s,l}{}\\
\node{}\node{0} \node{0} \node{0} \end{diagram}  \label{triangle3}
\end{equation}%
the columns split . Here $(N_{0},g_{0})$, $(E_{0},h_{0})$, $(M_{0},f_{0})$
are the kerneles of the maps $g$, $h$ and $f$, respectively.

The collection\ \  $\mathcal{S}$\ \ gives rise\ \ to a subfunctor $F$ \ \ of the additive
bifunctor $\mathrm{Ext}_{\Gamma }^{1}(-,-):(\mathrm{mod}\Gamma )\times (%
\mathrm{mod}\Gamma )^{op}\rightarrow \mathbf{Ab}$. The category $\mathrm{maps}(%
\mathrm{mod}\Lambda )$ has enough $F$-projectives and enough $F$-injectives,
the $F$-projectives are the maps of the form $M\xrightarrow{1_M}M$ and $%
0\rightarrow M$, and the $F$-injectives are of the maps of the form $M%
\xrightarrow{1_M}M$ and $M\rightarrow 0$.

According to Theorem \ref{maps2} we have the following:

\begin{theorem}
Classical tilting subcategories in $\mathrm{mod}(\mathrm{mod}\Lambda )$
correspond under $\Psi $ with relative tilting subcategories $\mathcal{T}_{%
\mathrm{mod}\Lambda }$ of $\mathrm{maps}(\mathrm{mod}\Lambda )$, such that
the following statements hold:

\begin{itemize}
\item[(i)] The maps $f:T_{0}\xrightarrow{f}T_{1}$ of objects in $\mathcal{T}%
_{\mathrm{mod}\Lambda }$ are monomorphisms.

\item[(ii)] Given $T:f:T_{0}\xrightarrow{f}T_{1}$ and $T^{\prime
}:g:T_{0}^{\prime }\xrightarrow{g}T_{1}^{\prime }$ in $\mathcal{T}_{\mathcal{%
C}}$. Then $\mathrm{Ext}_{F}^{1}(T,T^{\prime })=0$.

\item[(iii)] For each object $C$ in $\mathcal{C}$, there exist a exact
sequence in $\mathrm{maps}(\mathrm{mod}\Lambda )$:

\begin{equation*}
\begin{diagram}\dgARROWLENGTH=1em \node{0}\arrow{s} \node{0}\arrow{s}\\
\node{0}\arrow{s,l}{}\arrow{e,t}{} \node{C}\arrow{s,l}{}\\
\node{T_0}\arrow{s,l}{}\arrow{e,t}{f} \node{T_1}\arrow{s,l}{}\\
\node{T_0}\arrow{e,t}{g}\arrow{s} \node{T_1^{\prime}}\arrow{s}\\
\node{0} \node{0} \end{diagram}
\end{equation*}

such that the second column splits and $T:f:T_{0}\rightarrow T_{1}$, $%
T^{\prime }:g:T_{0}\rightarrow T_{1}^{\prime }$ are in $\mathcal{T}_{%
\mathcal{C}}$.
\end{itemize}
\end{theorem}

Since $\mathrm{gdim}(\mathrm{mod}\Lambda )\leq 2$, the relative global
dimension of $\mathrm{maps}(\mathrm{mod}\Lambda )$ is $\leq 2$.

For generalized tilting subcategories of $\mathrm{maps}(\mathrm{mod}\Lambda )
$ there is the following analogous to the previous theorem:

\begin{theorem}
Generalized tilting subcategories of $\mathrm{mod}(\mathrm{mod}\Lambda )$
correspond under $\Psi $ with relative tilting subcategories $\mathcal{T}_{%
\mathrm{mod}\Lambda }$ of $\mathrm{maps}(\mathrm{mod}\Lambda )$ such that
the following statements hold:

\begin{itemize}
\item[(i)] Given $T:T_{1}\rightarrow T_{0}$, $T^{\prime }:T_{1}^{\prime
}\rightarrow T_{0}^{\prime }$ in $\mathcal{T}_{\mathrm{mod}\Lambda }$. Then $%
\mathrm{Ext}_{F}^{k}(T,T^{\prime })=0$, for $0<k\leq 2.$

\item[(ii)] For each \ \ object $C$\ \  in $\mathcal{C}$, \ \ there exists a relative
exact sequence in $\mathrm{maps}(\mathrm{mod}\Lambda )$:
\begin{equation*}
0\rightarrow (0,C,0)\rightarrow T^{0}\rightarrow T^{1}\rightarrow
T^{2}\rightarrow 0
\end{equation*}%
with $T^{i}\in \mathcal{T}_{\mathrm{mod}\Lambda }$.
\end{itemize}
\end{theorem}

\subsection{Contravariantly Finite Categories in $\mathrm{mod}(\mathrm{mod}
\varLambda)$}

In this subsection we will see that some properties like: contravariently,
covariantly, functorially finite subcategories of $\mathrm{maps}(\mathrm{mod}%
\Lambda ),$ are preserved by the functor $\varPhi$.

The following theorem was proved in [18]. [See also 4 Theo. 5.5].

\begin{theorem}
Let $\mathcal{C}$ be a dualizing Krull-Schmidt variety. The assignments $%
\mathcal{T}\mapsto \mathcal{T}^{\bot }$ and $\mathscr Y\mapsto \mathscr %
Y\cap ^{\bot }\mathscr Y$ induce a bijection between equivalence classes of
tilting subcategories $\mathcal{T}$ of $\mathrm{mod}(\mathcal{C})$, with $%
\mathrm{pdim}\mathcal{T}\leq n$, such that $\mathcal{T}$ is a generator of $%
\mathcal{T}^{\bot }$ and classes of subcategories $\mathscr Y$ of $\mathrm{%
mod}(\mathcal{C})$ which are covariantly finite, coresolving, and whose
orthogonal complement $^{\bot }\mathscr Y$ has projective dimension $\leq n$.
\end{theorem}

Of course, the dual of the above theorem is true. Hence it is clear the
importance of studying; covariantly, contravariantly and functorially finite
subcategories in $\mathrm{mod}(\mathcal{C})$. We are specially interested in
the case $\mathcal{C}$ is the category of finitely generated left modules
over an artin algebra $\Lambda .$ In this situation we can study them via
the functor $\Psi $ relating them with the corresponding subcategories of $%
\mathrm{maps}(\mathrm{mod}\Lambda )$, which in principle are easier to
study, since $\mathrm{maps}(\mathrm{mod}\Lambda )$ and the category of
finitely generated left $\Gamma $- modules, with $\Gamma $ the triangular
matrix ring, are equivalent.

Such is the content of our next theorem.

\begin{theorem}
Let $\mathscr C\subset \mathrm{maps}(\mathrm{mod}\Lambda )$ be a
subcategory. Then the following statements hold:

\begin{itemize}
\item[(a)] If $\mathscr C$ is contravariantly finite in $\mathrm{maps}(%
\mathrm{mod}\Lambda )$, then $\varPhi(\mathscr C)$ is a contravariantly
finite subcategory of $\mathrm{mod}(\mathrm{mod}\Lambda )$.

\item[(b)] If $\mathscr C$ is covariantly finite in $\mathrm{maps}(\mathrm{%
mod}\Lambda )$, then $\varPhi(\mathscr C)$ is a covariantly finite
subcategory of $\mathrm{mod}(\mathrm{mod}\Lambda )$.

\item[(c)] If $\mathscr C$ is functorially finite in $\mathrm{maps}(\mathrm{%
mod}\Lambda )$, then $\varPhi(\mathscr C)$ is a functorially finite
subcategory of $\mathrm{mod}(\mathrm{mod}\Lambda )$.
\end{itemize}
\end{theorem}

\begin{proof}
(a) Assume $\mathscr C\subset \mathrm{maps}(\mathrm{mod}\Lambda )$ is
contravariantly finite. Let $F$ be a functor in $\mathrm{mod}(\mathrm{mod}\Lambda ))$
and $(\;,M_{1})\xrightarrow{(\;,f)}(\;,M_{2})\rightarrow F\rightarrow 0$ a
minimal projective presentation of $F$. Then, there exist a map $Z:Z_{1}%
\xrightarrow{h}Z_{2}$ and a map
\begin{equation*}
\begin{diagram}\dgARROWLENGTH=1em\label{contramaps1}
\node{Z_1}\arrow{e,t}{h}\arrow{s,l}{q_1} \node{Z_2}\arrow{s,l}{q_2}\\
\node{M_1}\arrow{e,t}{f} \node{M_2} \end{diagram}
\end{equation*}

in $\mathrm{maps}(\mathrm{mod}\Lambda )$ such that $Z=(Z_{1},Z_{2},h)$ is a
right $\mathscr C$-approximation of $M=(M_{1},M_{2},f)$. The diagram (\ref%
{contramaps1}) induce the following commutative exact diagram:
\begin{equation*}
\begin{diagram}\dgARROWLENGTH=1em
\node{(\;,Z_1)}\arrow{e,t}{(\;,h)}\arrow{s,l}{(\;,q_1)}
\node{(\;,Z_2)}\arrow{e,t}{}\arrow{s,l}{(\;,q_2)}
\node{\varPhi(Z)}\arrow{e}\arrow{s,l}{\rho} \node{0}\\
\node{(\;,M_1)}\arrow{e,t}{(\;,f)} \node{(\;,M_2)}\arrow{e,t}{}
\node{F}\arrow{e} \node{0} \end{diagram}
\end{equation*}

We claim that $\rho $ is a right $\varPhi(\mathscr{C})$-approximation of $F$%
. Let $H\in \varPhi(\mathscr C)$, $\eta :H\rightarrow F$ a map and $%
(\;,X_{1})\xrightarrow{(\;,r)}(\;,X_{2})\rightarrow H\rightarrow 0$ a
minimal projective presentation of $H$. We have a lifting of $\eta :$
\begin{equation*}
\begin{diagram}\dgARROWLENGTH=1em
\node{(\;,X_1)}\arrow{e,t}{(\;,r)}\arrow{s,l}{(\;,s_1)}
\node{(\;,X_2)}\arrow{e,t}{}\arrow{s,l}{(\;,s_2)}
\node{H}\arrow{e}\arrow{s,l}{\eta} \node{0}\\
\node{(\;,M_1)}\arrow{e,t}{(\;,f)} \node{(\;,M_2)}\arrow{e,t}{}
\node{F}\arrow{e} \node{0} \end{diagram}
\end{equation*}

By Yoneda's Lemma, there is the following commutative square:
\begin{equation*}
\begin{diagram}\dgARROWLENGTH=1em \node{X_1}\arrow{e,t}{r}\arrow{s,l}{s_1}
\node{X_2}\arrow{s,l}{s_2}\\ \node{M_1}\arrow{e,t}{f} \node{M_2}
\end{diagram}
\end{equation*}%
with $X=(X_{1},X_{2},r)\in \mathscr C$. Since $Z=(Z_{1},Z_{2},h)$ is a right
$\mathscr C$-approximation of $M=(M_{1},M_{2},f)$, there exists a morphism $%
(t_{1},t_{2}):(X_{1},X_{2},r)\rightarrow (Z_{1},Z_{2},h)$, such that the
following diagram:
\begin{equation*}
\begin{diagram}\dgARROWLENGTH=1em \node{X_1}\arrow{e,t}{r}\arrow{s,l}{t_1}
\node{X_2}\arrow{s,l}{t_2}\\ \node{Z_1}\arrow{e,t}{h}\arrow{s,l}{q_1}
\node{Z_2}\arrow{s,l}{q_2}\\ \node{M_1}\arrow{e,t}{f} \node{M_2}
\end{diagram}
\end{equation*}%
is commutative, with $q_{i}t_{i}=s_{i}$ for $i=1,2$. Which implies the
existence of a map $\theta :H\rightarrow \Psi (H)=G$, such that the diagram
\begin{equation*}
\begin{diagram}\dgARROWLENGTH=1em
\node{(\;,X_1)}\arrow{e,t}{(\;,r)}\arrow{s,l}{(\;,t_1)}
\node{(\;,X_2)}\arrow{e}\arrow{s,l}{(\;,t_2)}
\node{H}\arrow{e}\arrow{s,l}{\theta} \node{0}\\
\node{(\;,Z_1)}\arrow{e,t}{(\;,h)}\arrow{s,l}{(\;,q_1)}
\node{(\;,Z_2)}\arrow{e}\arrow{s,l}{(\;,q_2)}
\node{G}\arrow{e}\arrow{s,l}{\rho} \node{0}\\
\node{(\;,M_1)}\arrow{e,t}{(\;,f)} \node{(\;,M_2)}\arrow{e}
\node{F}\arrow{e} \node{0} \end{diagram}
\end{equation*}%
with exact raws, is commutative, this is: $\rho \theta =\eta $.

The proof of (b) is dual to (a), and (c) follows from (a) and (b).
\end{proof}

\begin{theorem}
Let $\mathscr{C}\subset \mathrm{maps}(\mathrm{mod}\Lambda )$ be a category
which contains the objects of the form $(M,0,0)$, $(M,M,1_{M})$, and assume $%
\varPhi(\mathscr{C})$ is contravariantly finite. Then $\mathscr{C}$ is
contravariantly finite.
\end{theorem}

\begin{proof}
Let $M_{1}\xrightarrow{f}M_{2}$ a map in $\mathrm{maps}(\mathrm{mod}\Lambda
) $, then we have an exact sequence $(\;,M_{1})\xrightarrow{(\;,f)}%
(\;,M_{2})\rightarrow F\rightarrow 0$.

There exist $G\in \varPhi(\mathcal{C})$ such that $\rho :G\rightarrow F$ is
a right $\varPhi(\mathscr{ C})$-approximation.

Let
\begin{equation*}
(\;,Z_{1})\xrightarrow{(\;,h)}(\;,Z_{2})\rightarrow G\rightarrow 0
\end{equation*}%
be a minimal projective presentation of $G$. Then, there exists a map $%
(r_{1},r_{2}):(Z_{1},Z_{2},h)\rightarrow (M_{1},M_{2},f)$ such that%
\begin{equation*}
\begin{diagram}\dgARROWLENGTH=1em
\node{(\;,Z_1)}\arrow{e,t}{(\;,h)}\arrow{s,l}{(\;,r_1)}
\node{(\;,Z_2)}\arrow{e}\arrow{s,l}{(\;,r_2)}
\node{G}\arrow{e}\arrow{s,l}{\rho} \node{0}\\
\node{(\;,M_1)}\arrow{e,t}{(\;,f)} \node{(\;,M_2)}\arrow{e}
\node{F}\arrow{e} \node{0} \end{diagram}
\end{equation*}%
is a lifting of $\rho $.

Let $(X_{1},X_{2},g)$ be an object in $\mathscr{C}$ and a map $(v_1,v_2):(X_{1},X_{2},g)\rightarrow (M_1,M_2,f)$,
which \ induces \ the \ following \ commutative exact diagram in\ \ $\mathrm{mod}%
(\mathrm{mod}\Lambda):$
\begin{equation*}
\begin{diagram}\dgARROWLENGTH=1em
\node{(\;,X_1)}\arrow{e,t}{(\;,g)}\arrow{s,l}{(\;,v_1)}
\node{(\;,X_2)}\arrow{s,l}{(\;,v_2)}\arrow{e}
\node{H}\arrow{e}\arrow{s,l}{\eta} \node{0} \\
\node{(\;,M_1)}\arrow{e,t}{(\;,f)} \node{(\;,M_2)}\arrow{e}
\node{F}\arrow{e} \node{0} \end{diagram}
\end{equation*}

Since $\rho :G\rightarrow F$ is a right $\varPhi(\mathscr{C})$%
-approximation, there exists a morphism $\theta :H\rightarrow G$ such taht $%
\rho \theta =\eta $. Therefore: $\theta $ induces a morphism
\begin{equation*}
\begin{diagram}\dgARROWLENGTH=1em \node{X_1}\arrow{e,t}{g}\arrow{s,l}{t_1}
\node{X_2}\arrow{s,l}{t_2}\\ \node{Z_1}\arrow{e,t}{h} \node{Z_2}
\end{diagram}
\end{equation*}
in $\mathrm{maps}(\mathcal{C})$ such that $\varPhi(t_{1},t_{2})=\theta $.

We have two liftings of $\rho :$
\begin{equation*}
\begin{diagram}\dgARROWLENGTH=1.5em \node{0}\arrow{e}
\node{(\;,X_0)}\arrow{e,t}{g_0}\arrow{s,l}{(\;,v_0)}\arrow{s,r}{(\;,r_0t_0)}
\node{(\;,X_1)}\arrow{e,t}{(\;,g)}\arrow{s,l}{(\;,v_1)}\arrow{s,r}{(%
\;,r_1t_1)}
\node{(\;,X_2)}\arrow{s,l}{(\;,v_2)}\arrow{e}\arrow{s,r}{(\;,r_2t_2)}
\node{H}\arrow{e}\arrow{s,l}{\rho} \node{0} \\ \node{0}\arrow{e}
\node{(\;,M_0)}\arrow{e,t}{f_0} \node{(\;,M_1)}\arrow{e,t}{(\;,f)}
\node{(\;,M_2)}\arrow{e} \node{F}\arrow{e} \node{0} \end{diagram}
\end{equation*}%
Then they are homotopic, and there exist maps $(\;,\lambda
_{1}):(\;,X_{1})\rightarrow (\;,M_{0})$ and $(\;,\lambda
_{2}):(\;,X_{2})\rightarrow (\;,M_{1})$, such that
\begin{eqnarray*}
v_{1} &=&r_{1}t_{1}+f_{0}\lambda _{1}+\lambda _{2}g, \\
v_{2} &=&r_{2}t_{2}+f\lambda _{2}.
\end{eqnarray*}

We have the following commutative diagram:
\begin{equation*}
\begin{diagram}\dgARROWLENGTH=1em \node{X_1}\arrow{e,t}{g}\arrow{s,l}{m_1}
\node{X_2}\arrow{s,l}{m_2}\\ \node{Z_1\coprod M_0\coprod
M_1}\arrow{e,t}{w}\arrow{s,l}{n_1} \node{Z_2\coprod M_1}\arrow{s,l}{n_2}\\
\node{M_1}\arrow{e,t}{f} \node{M_2} \end{diagram}
\end{equation*}%
with morphisms $n_{1}=(r_{1}\;f_{0}\ 1_{M_{1}}),\;n_{2}=(r_{2}\;f_{2})$ and
\begin{equation*}
m_{1}=\left(
\begin{array}{c}
t_{1} \\
\lambda _{1} \\
\lambda _{2}g%
\end{array}%
\right) ,m_{2}=\left(
\begin{array}{c}
t_{2} \\
\lambda _{2}%
\end{array}%
\right) ,w=\left(
\begin{array}{ccc}
h & 0 & 0 \\
0 & 0 & 1_{M_{1}}%
\end{array}%
\right) .
\end{equation*}%
But $w:Z_{1}\coprod M_{0}\coprod M_{1}\rightarrow Z_{2}\coprod M_{1}$ is in $%
\mathscr C,$ and
\begin{equation*}
\begin{diagram}\dgARROWLENGTH=1em \node{Z_1\coprod M_0\coprod
M_1}\arrow{e,t}{w}\arrow{s,l}{n_1} \node{Z_2\coprod M_1}\arrow{s,l}{n_2}\\
\node{M_1}\arrow{e,t}{f} \node{M_2} \end{diagram}
\end{equation*}%
is a right $\mathscr{C}$-approximation of $M_{1}\xrightarrow{f}M_{2}$.
\end{proof}

We can define the functor $\varPhi^{op}:\mathrm{maps}(\mathrm{mod}\Lambda
)\rightarrow \mathrm{mod(}(\mathrm{mod}\Lambda \mathcal{)}^{op})$ as:
\begin{equation*}
\varPhi^{op}(A_{1}\xrightarrow{f}A_{0})=\mathrm{Coker}((A_{0},-)%
\xrightarrow{(f,-)}(A_{1},-))
\end{equation*}%
We have the following dual of the above theorem, whose proof we leave to the
reader:

\begin{theorem}
Let $\mathscr{C}\subset \mathrm{maps}(\mathrm{mod}\Lambda )$ be a
subcategory that contains the objects of the form $(0,M,0)$ and $(M,M,1_{M})$%
. If $\varPhi^{op}(\mathscr C)$ is contravariantly finite, then $\mathscr{C}$
is covariantly finite.
\end{theorem}

\begin{definition}
The subcategory $\mathscr{C}$ of $\mathrm{maps}(\mathrm{mod}\Lambda )$
consisting\ \ of \ \ all \ \ maps $(M_{1},M_{2},f)$, such that $f$ is an epimorphism,
will be called the category of epimaps, $\mathrm{epimaps}(\mathrm{mod}%
\Lambda )$. Dually, the subcategory $\mathscr{C}$ of $\mathrm{maps}(\mathrm{%
mod}\Lambda )$ consisting of all maps $(M_{1},M_{2},f)$, such that\ \ $f$ is a
monomorphism, will \ \ be called \ \ the \ \ category of monomaps, $\mathrm{monomaps}(%
\mathrm{mod}\Lambda ).$
\end{definition}

We have the following examples of functorially finite subcategories of the
category $%
\mathrm{maps}(\mathrm{mod}\Lambda )$:

\begin{proposition}
The categories $\mathrm{epimaps}(\mathrm{mod}\Lambda )$ and $\mathrm{monomaps%
}(\mathrm{mod}\Lambda )$ are functorially finite in $\mathrm{maps}(\mathrm{mod%
}\Lambda ).$
\end{proposition}

\begin{proof}
Let $M_{1}\xrightarrow{f}M_{2}$ be an object in $\mathrm{maps}(\mathrm{mod}%
\Lambda )$. Then we have the following right approximation
\begin{equation*}
\begin{diagram}\dgARROWLENGTH=1em \node{M_1}\arrow{e,t}{f'}\arrow{s,l,=}{}
\node{\mathrm{Im}(f)}\arrow{e,t}{} \arrow{s,l}{j} \node{0}\\
\node{M_1}\arrow{e,t}{f} \node{M_2} \end{diagram}
\end{equation*}

Consider an epimorphism $X_{2}\xrightarrow{g}X_{2}\rightarrow 0$ and a map
in $\mathrm{maps}(\mathrm{mod}\Lambda )$ $(t_{1},t_{2}):(X_{1},X_{2},g)%
\rightarrow (M_{1},M_{2},f)$. Then we have the following commutative
diagram:
\begin{equation*}
\begin{diagram}\dgARROWLENGTH=1em \node{X_1}\arrow{e,t}{g}\arrow{s,l}{t_1}
\node{X_2}\arrow{e,t}{} \arrow{s,l}{t_2} \node{0}\arrow{s} \node{}\\
\node{M_1}\arrow{e,t}{f} \node{M_2}\arrow{e,t}{\pi}
\node{\mathrm{Coker}(f)}\arrow{e} \node{0} \end{diagram}
\end{equation*}%
Since $\pi t_{2}=0$, the map $t_{2}:X_{2}\rightarrow M_{2}$ factors through $%
j:\mathrm{Im}(f)\rightarrow M_{2}$, this is: there is a map $%
u:X_{2}\rightarrow \mathrm{Im}(f)$ such that $ju=t_{2}$, and we have the
commutative diagram:
\begin{equation*}
\begin{diagram}\dgARROWLENGTH=1em \node{X_1}\arrow{e,t}{g}\arrow{s,l}{t_1}
\node{X_2}\arrow{e,t}{} \arrow{s,l}{u} \node{0}\\
\node{M_1}\arrow{e,t}{f'}\arrow{s,l,=}{}
\node{\mathrm{Im}(f)}\arrow{s,l}{j}\\ \node{M_1}\arrow{e,t}{f} \node{M_2}
\end{diagram}
\end{equation*}

Now, let $P\xrightarrow{p}M_{2}\rightarrow 0$ be the projective cover of $%
M_{2}$. We get a commutative exact diagram:
\begin{equation}
\begin{diagram}\label{epimap}
\dgARROWLENGTH=1em
\node{M_1}\arrow{e,t}{f}\arrow{s,l}{(^1_0)} \node{M_2}
\arrow{s,l}{1_{M_2}}\\ \node{M_1\oplus P}\arrow{e,t}{[f\; p]}	
\node{M_2}\arrow{e} \node{0} \end{diagram}
\end{equation}

Let $X_{1}\xrightarrow{g}X_{2}\rightarrow 0$ be an epimorphism and
\begin{equation*}
\begin{diagram}\dgARROWLENGTH=1em \node{M_1}\arrow{e,t}{f}\arrow{s,l}{s_1}
\node{M_2} \arrow{s,l}{s_2}\\ \node{X_1}\arrow{e,t}{g} \node{X_2}\arrow{e}
\node{0} \end{diagram}
\end{equation*}

a map of maps.

Since $P$ is projective, we have the following commutative square
\begin{equation*}
\begin{diagram}\dgARROWLENGTH=1em \node{P}\arrow{e,t}{p}\arrow{s,l}{\mu}
\node{M_2} \arrow{s,l}{s_2}\\ \node{X_1}\arrow{e,t}{g} \node{X_2}\arrow{e}
\node{0} \end{diagram}
\end{equation*}%
Then gluing the two squares:
\begin{equation*}
\begin{diagram}\dgARROWLENGTH=1em
\node{M_1}\arrow{e,t}{f}\arrow{s,l}{(^1_0)} \node{M_2}
\arrow{s,l,=}{}\\ \node{M_1\oplus P}\arrow{e,t}{[f\;p]}\arrow{s,l}{[s_1\;
\mu]} \node{M_2}\arrow{s,l}{s_2}\arrow{e,t}{} \node{0}\\
\node{X_1}\arrow{e,t}{g} \node{X_2}\arrow{e} \node{0} \end{diagram}
\end{equation*}

we obtain the map $(s_{1},s_{2})$.

Then (\ref{epimap}) is a left approximation. The second part is dual.
\end{proof}

\begin{corollary}

\begin{itemize}
\item[(i)] The category $\mathrm{mod(}\Lambda )^{O}$ of functors vanishing
on projectives is functorially finite.

\item[(ii)] The category $\hat{\mathrm{mod(}\Lambda )}$ of functors with $%
\mathrm{pd}\leq 1$ is functorially finite.
\end{itemize}
\end{corollary}

\begin{proof}
The proof of this follows immediately from
\begin{equation*}
\mathrm{mod(}\Lambda )^{O}=\varPhi(\mathrm{epimaps}(\mathrm{mod}\Lambda ),%
\hat{\mathrm{mod(}\Lambda )}=\varPhi(\mathrm{monomaps}(\mathrm{mod}\Lambda ).
\end{equation*}
\end{proof}

In view of the previous theorem it is of special interest to characterize
the functors in $\mathrm{\mathrm{mod}}(\mathrm{mod}\Lambda )$ of projective
dimension less or equal to one.

The radical $t_{H}(F)$ of a finetly presented functor $F$, is defined as $%
t_{H}(F)=\underset{L\in \Theta }{\Sigma }L$, where $\Theta $ is the
collection of subfunctors of $F$ of finite length and with composition
factors the simple objects of the form $S_{M}$, with $M$ a non projective
indecomposable module.

\begin{definition}
Let $F$ be a finitely presented functor. Then $F$ is torsion free if and
only if $t_{H}(F)=0.$
\end{definition}

We end the paper with the following result, whose proof is essentially in
[19, Lemma 5.4.]

\begin{lemma}
Let $F$ be a functor in $\mathrm{\mathrm{mod}}(\mathrm{mod}\Lambda )$. Then $F$
has projective dimension less or equal to one, if and only if, $F$ is
torsion free.
\end{lemma}

{\bf ACKNOWLEDGEMENTS.}
The second author thanks CONACYT for giving him financial support
during his graduate studies.

This paper is in final form and no version of it will be submitted
for publication elsewhere.


\begin{thebibliography}{MVO1}

\bibitem{1}{M. Auslander, Representation Dimension of Artin
Algebras I,  Queen Mary College Mathematics Notes, 1971., \em Selected Works of
Maurice Auslander, Part 1 AMS,} 505 - 574.


\bibitem{2}{M. Auslander, Representation Theory of Artin Algebras I, \em  Communications In Algebra, }{\bf 1(3)} (1974), 177 - 268.

\bibitem{3}{M. Auslander and  Ragnar-Olaf Buchweitz, The Homological Theory of Maximal Cohen-Macaulay Approximations, \em  Soci\'{e}te
Math\'{e}matique de France M\'{e}moire, }{\bf  38} (1989), 5 - 37.


\bibitem{4}{M.  Auslander and I. Reiten, Applications of Contravariantly Finite Subcategories, \em  Adavances in
Mathematics, }{\bf 86(1)} (1991),  111 - 151.



\bibitem{5}{M.  Auslander and I.  Reiten,  Stable equivalence of
Artin algebras, Proceedings of the Conference on Orders, Group Rings and
Related Topics (Ohio State Univ., Columbus, Ohio, 1972), \em  Lecture
Notes in Math. Springer, Berlin, }{\bf 353} (1973),  8 - 71.



\bibitem{6}{M. Auslander and I. Reiten, Stable Equivalence of Dualizing $R$-Varietes, \em  Adavances in Mathematics, }{\bf  12(3)}  (1974),  306 - 366.






\bibitem{7}{M. Auslander, I. Reiten and S. O.  Smal\o {}, Representation Theory of Artin Algebras, \em  Cambridge University Press, }{\bf 36} (1995).




\bibitem{8}{M. Auslander and  \O {}. Solberg, Relative homology
and repres
entation theory I. Relative cotilting theory, \em Communications in Algebra, }{\bf 21(9)} (1993), 3033 - 3097.


\bibitem{9}{M. Auslander and  \O {}. Solberg, Relative homology
and representation theory I. Relative homology and homologically finite
subcategories, \em Communications in Algebra, }{\bf 21(9)} (1993), 2995 - 3031.







\bibitem{10}{K. Bongartz, Tilted algebras,  M. Auslander (ed.) E.
Lluis (ed.), Representations of Algebras. Proc. ICRA III, \em Lecture Notes in
Mathematics,  Springer, }{\bf 903} (1981), 26 - 38.




\bibitem{11}{S.  Brenner and M. C. R  Butler, Generalisations of the
Bernstein-Gelfand-Ponomarev reflection functors, in Proc. ICRA II (Ottawa,
1979), \em Lecture Notes in Math.  Springer-Verlag, Berlin, Heidelberg,
New York, }{\bf 832} (1980),  103 - 169.

\bibitem{12}{E. Cline, B. Parshall and L. Scott, Derived categories and Morita theory, \em  J. Algebra, } {\bf 104} (1986),
397 - 409.

\bibitem{13}{R.  Fossum, H. B.  Foxby,  P.  Griffith and I.  Reiten, Minimal injective resolutions with applications to dualizing modules
and Gorenstein modules, \em  Inst. Hautes \'{E}tudes Sci. Publ. Math., }{\bf  45} (1975), 193 - 215.

\bibitem{14}{D. Happel, Triangulated Categories in the
Representation Theory of Finite Dimensional Algebras, \em London Math. Soc.
Lecture Note Soc., }{\bf 119} (1988).


\bibitem{15}{R. Mart\'{\i}nez-Villa, Applications of Koszul
algebras: The preprojective algebra, \em  Representation Theory of Algebras, CMS
Conf. Proc.  Amer. Math. Soc., Providence, RI, }{\bf 18} (1996),  487 - 504.

\bibitem{16}{R. Mart\'{\i}nez-Villa  and G. Monta\~{n}o-Berm\'{u}dez, Triangular matrix and Koszul algebras, \em Int. J. Algebra, }{\bf 1, 9 - 12}
(2007), 441 - 467.

\bibitem{17}{R.  Mart\'{\i}nez-Villa and M.  Ortiz, Tilting Theory
and Functor Categories I. Classical Tilting, \em preprint, arXiv:math.RT/0165593.}

\bibitem{18}{R.  Mart\'{\i}nez-Villa and M.  Ortiz, Tilting Theory
and Functor Categories II. Generalized Tilting, \em preprint, arXiv:math.RT/0165605.}

\bibitem{19}{R.  Mart\'inez-Villa and \O{}. Solberg,  Graded
and Koszul categories, \em  Appl. Categ. Structures,
doi:10.1007/s10485-009-9191-6, in press.}

\bibitem{20}{R.  Mart\'inez-Villa and \O{}. Solberg, Artin
and Schelter regular algebras and categories, \em Journal of Pure and Applied
Algebra, in press.}

\bibitem{21}{R.  Mart\'inez-Villa and \O{}. Solberg, Noetherianity and Gelfand-Kirillov dimension of components, \em J. Algebra, }{\bf 323 (5)} (2010),
 1369 - 1407.

\bibitem{22}{R.  Mart\'inez-Villa and \O{}. Solberg, Serre
duality for Artin-Schelter regular $K$-categories, \em  Int. J. Algebra, }{\bf 3(5-9)} (2009), 355 - 375.

\bibitem{23}{J.  Miyachi, Derived Categories With
Applications to Representations of Algebras, www.u-gakugei.ac.jp/\~{}miyachi/papers/ChibaSemi.pdf.}

\end{thebibliography}
\end{document}